\documentclass{amsart}
\pdfoutput=0
\pdfoutput=1
\usepackage[osf]{mathpazo}
\usepackage{pb-diagram}
\usepackage{amscd}
\usepackage{amssymb}
\usepackage{mathptmx}
\usepackage{enumerate}

\newtheorem{theorem}{Theorem}[section]

\newtheorem{proposition}[theorem]{Proposition}

\theoremstyle{definition}

\newtheorem{question}[theorem]{Question}
\newtheorem{remark}[theorem]{Remark}

\DeclareMathOperator{\B}{B}
\DeclareMathOperator{\Cl}{Cl}

\DeclareMathOperator{\Div}{div}
\DeclareMathOperator{\HH}{H}
\DeclareMathOperator{\GG}{G}
\DeclareMathOperator{\Pic}{Pic}
\DeclareMathOperator{\Proj}{Proj}
\DeclareMathOperator{\Spec}{Spec}

\title[Brauer Group of an Affine Double Plane]{The Brauer Group of
  an Affine Double Plane Associated to a Hyperelliptic Curve}
\author{Timothy J. Ford}%
\address{Department of Mathematics, Florida
  Atlantic University, Boca Raton, Florida 33431}%
\date{\today}%
\email{\tt ford@fau.edu}%
\thanks{Preliminary Report}%
\begin{document}

\begin{abstract}
  For an affine double plane defined by an equation of the form $z^2 =
  f$ we study the divisor class group and the Brauer group.  Two cases
  are considered. In the first case, $f$ is a product of $n$ linear
  forms in $k[x,y]$ and $X$ is birational to a ruled surface $\mathbb
  P^1 \times C$, where $C$ is rational if $n$ is odd and hyperelliptic
  if $n$ is even.  In the second case, $f$ is the equation of an
  affine hyperelliptic curve.  On the open set where the cover is
  unramified, we compute the groups of divisor classes, the Brauer
  groups, the relative Brauer group, as well as all of the terms in
  the exact sequence of Chase, Harrison and Rosenberg.
\end{abstract}

\subjclass[2010]{16K50; Secondary
  14F22, 14J26, 14C20}
\keywords{Brauer group, Class group, algebraic surface}

\maketitle

\section{Introduction}
\label{sec:1}

\subsection{Statement of main results}
\label{sec:1.1}
Throughout $k$ is an algebraically closed field of characteristic
different from $2$.  This article is concerned with the study of
divisor classes and algebra classes on a double cover $X$ of the
affine plane $\mathbb A^2$.  The surface $X$ which we investigate is a
hypersurface in $\mathbb A^3$ which is defined by an equation of the
form $z^2 =f$.

In Section~\ref{sec:2}, we investigate the surface $X$ which is
defined by an equation of the form $z^2 = f_1\dotsm f_n$, where each
$f_i$ is a linear form in $k[x,y]$.  The surface $X$ is birational to
a ruled surface $\mathbb P^1 \times C$, where $C$ is rational if $n$
is odd and hyperelliptic if $n$ is even. There is an isolated singular
point at the origin.

In Section~\ref{sec:4}, the surface $X$ is defined by an equation of
the form $z^2 = f$, where $f = y^2 - p(x)$.  The zero set of $f$ in
$\mathbb A^2$ is denoted $F$. If $p(x)$ has degree at least three and
is square-free, then $F$ is an affine hyperelliptic curve.  The
surface $X$ is rational, normal, and any singularity on $X$ is a
rational double point of type $A_n$.

The invariants of the surface $X$ which we compute include the group
of Weil divisor classes $\Cl(\cdot)$, the Picard group $\Pic(\cdot)$,
and the Brauer group $\B(\cdot)$.  Let $K \rightarrow L$ be the
quadratic Galois extension with group $G$ of the rational function
fields of $\mathbb A^2$ and $X$.  On an open affine subset, the double
cover $X \rightarrow \mathbb A^2$ is unramified, let $R \rightarrow S$
be the corresponding quadratic Galois extension of commutative rings.
We are especially interested in the relative Brauer group $\B(S/R)$
and the image of $\B(R) \rightarrow \B(L)$.  For the Galois extension
$R\rightarrow S$, the terms in the Chase, Harrison, Rosenberg
cohomology sequence \cite[Corollary~5.5]{CHR:GtGccr} are computed.  It
is shown that the relative Brauer group $\B(S/R)$ is isomorphic to
$\HH^1(G,\Cl(X))$. In \cite{F:rBgadp} this relationship is
investigated further, where it is demonstrated that every element of
$\B(S/R)$ is represented by a generalized crossed product algebra of
degree two over $k[x,y]$. 

To simplify the exposition, in Section~\ref{sec:3} we include some
computations involving divisors on a hyperelliptic curve.

\subsection{Background material}
\label{sec:1.2}
We recommend \cite{M:EC} as a source for all unexplained notation and
terminology.  We tacitly assume all groups and sequences of groups are
`modulo the characteristic of $k$'.  By $\mu_d$ we denote the kernel
of the $d$th power map $k^\ast \rightarrow k^\ast$. By $\mu$ we denote
$\cup_d \mu_d$. There is an isomorphism $\mathbb Q/\mathbb Z \cong
\mu$, which is non-canonical and when convenient we use the two groups
interchangeably. Cohomology and sheaves are for the \'etale topology,
except where we use group cohomology.  For instance, $\mathcal O$
denotes the sheaf of rings and $\mathbb G_m$ denotes the sheaf of
units.  If $X$ is a variety over $k$, then the multiplicative group of
units on $X$ is $\HH^0(X, \mathbb G_m)$ which is also denoted
$\mathcal O^\ast(X)$, or simply $X^\ast$.  The Picard group
\cite[Proposition~III.4.9]{M:EC}, $\Pic(X)$, is given by $\HH^1(X,
\mathbb G_m)$. The torsion subgroup of $\HH^2(X, \mathbb G_m)$ is the
cohomological Brauer group \cite{G:GBI}, denoted $\B^\prime(X)$.  The
Brauer group $\B(X)$ of classes of $\mathcal O(X)$-Azumaya algebras
embeds into $\B'(X)$ in a natural way \cite[(2.1), p.  51]{G:GBI} and
for all varieties considered in this article $\B(X) = \B'(X)$
\cite{H:WBr}.  By Kummer theory, the $d$th power map
\begin{equation}
  \label{eq:101}
  1 \rightarrow \mu_d \rightarrow \mathbb G_m \xrightarrow{d} \mathbb G_m
  \rightarrow 1
\end{equation}
is an exact sequence of sheaves on $X$.  For any abelian group $M$ and
integer $d$, by ${_dM}$ we denote the subgroup of $M$ annihilated by
$d$.  The long exact sequence of cohomology associated to
\eqref{eq:101} breaks up into short exact sequences which in degrees
one and two are
\begin{gather}
  \label{eq:102}
  1 \rightarrow X^\ast / X^{\ast d} \rightarrow \HH^1(X,\mu_d)
  \rightarrow {_d\Pic{X}} \rightarrow 0\\
  \label{eq:103}
  0 \rightarrow \Pic{X} \otimes \mathbb Z/d \rightarrow \HH^2(X,\mu_d)
  \rightarrow {_d\B(X)} \rightarrow 0
\end{gather}
The group $\HH^1(X,\mu_d)$ classifies the Galois coverings  $Y
\rightarrow X$ with cyclic Galois group
$\mathbb Z/d$ \cite[pp. 125--127]{M:EC}.  We will utilize the following
form of the exact sequence of Chase, Harrison and Rosenberg
\cite[Corollary~5.5]{CHR:GtGccr}.
  If $Y \rightarrow X$ is a Galois covering with finite cyclic group
  $G$, and $\Pic{X} = 0$, then there is an exact sequence of abelian
  groups.
  \begin{equation}
    \label{eq:104}
    0 \rightarrow \left(\Pic{Y}\right)^G \rightarrow \HH^2(G, Y^\ast)
    \xrightarrow{\alpha_4} \B(Y/X) \xrightarrow{\alpha_5}
    \HH^{1}(G,\Pic{Y}) \rightarrow 0\ldotp
  \end{equation}
If $X$ is a nonsingular integral affine surface over $k$ with field of
rational functions $K=K(X)$, then there is an exact sequence
\begin{equation}
  \label{eq:105}
  0 \to \B(X) \to \B(K) \xrightarrow{a} \bigoplus_{C \in
    X_1}\HH^1(K(C),\mathbb Q/\mathbb Z) \xrightarrow{r} \bigoplus_{p
    \in X_2} \mu(-1) \xrightarrow{S} 0
\end{equation}
The first summation is over all irreducible curves $C$ on $X$, the
second summation is over all closed points $p$ on $X$.  Sequence
\eqref{eq:105} is obtained by combining sequences (3.1) and (3.2) of
\cite[p.~86]{AM:See}.  The map $a$ of \eqref{eq:105} is called the {\em
  ramification map}.  Let $\Lambda$ be a central $K$-division algebra
which represents a class in $\B(K)$. The curves $C \in X_1$ for which
$a([\Lambda])$ is non-zero make up the so-called {\em ramification
  divisor\/} of $\Lambda$ on $X$.  For $\alpha$, $\beta$ in $K^\ast$,
by $\Lambda = (\alpha,\beta)_d$ we denote the symbol algebra over $K$
of degree $d$.  Recall that $\Lambda$ is the associative $K$-algebra
generated by two elements, $u$ and $v$ subject to the relations $u^d =
\alpha$, $v^d=\beta$, $uv = \zeta vu$, where $\zeta$ is a fixed
primitive $d$th root of unity.  On the Brauer class containing
$(\alpha,\beta)_d$, the ramification map $a$ agrees with the so-called
{\em tame symbol}.  Let $C$ be a prime divisor on $X$.  Then
$\mathcal O_{X,C}$ is a discrete valuation ring with valuation denoted
by $\nu_C$. The residue field is $K(C)$, the field of rational
functions on $C$.  The ramification of $(\alpha,\beta)_d$ along $C$ is
the cyclic Galois extension of $K(C)$ defined by adjoining the $d$th
root of
\begin{equation}
  \label{eq:106}
  \alpha^{\nu_C(\beta)}
  \beta^{-\nu_C(\alpha)}\ldotp  
\end{equation}
If $(\alpha,\beta)_d$ has non-trivial ramification along $C$, then $C$
is a prime divisor of $(\alpha)$ or $(\beta)$.

In the usual way, we consider the affine plane $\mathbb A^2$ as an
open subset of the projective plane $\mathbb P^2$.  Let $F_0, F_1,
\dots, F_n$ be distinct curves in $\mathbb P^2$ each of which is
simply connected, that is, $\HH^1(F_i,\mathbb Q/\mathbb Z) = 0$.  Let
$Z = \{ F_i \cap F_j \mid i \neq j\}$, which is a subset of the
singular locus of $F = F_0 + F_1 + \dots+ F_n$.  Because each $F_i$ is
simply connected, if there is a singularity of $F$ which is not in
$Z$, then at that point $F$ is geometrically unibranched.  Decompose
$Z$ into irreducible components $Z_1 + \dots + Z_s$.  The {\em
  graph\/} of $F$ is denoted $\Gamma$ and is bipartite with edges
$\{F_0, \dots, F_n\} \cup \{Z_1, \dots, Z_s\}$. An edge connects $F_i$
to $Z_j$ if and only if $Z_i \subseteq F_i$.  So $\Gamma$ is a
connected graph with $n+1+s$ vertices. Let $e$ denote the number of
edges. Then $\HH_1(\Gamma, \mathbb Z/d) \cong {(\mathbb
  Z/d)}^{(r)}$, where $r = e - (n+1+s) + 1$.  We will utilize the
following special version of \cite[Theorem~4]{F:Bgk}.

\begin{theorem}
  \label{theorem:1.2}
  Let $f_1, \dots, f_n$ be irreducible polynomials in $k[x_1, x_2]$
  defining $n$ distinct curves $F_1$, \dots, $F_n$ in the projective
  plane $\mathbb P^2 = \Proj{k[x_0, x_1, x_2]}$.  Let $F_0 = Z(x_0)$
  be the line at infinity and $\Gamma$ the graph of $F = F_0+ F_1+
  \dots + F_n$.  If $\HH^1(F_i,\mathbb Q/\mathbb Z) = 0$ for each $i$,
  and $R = k[x_1, x_2][f_1^{-1}, \dots, f^{-1}_n]$, then for each $d
  \geq 2$, ${_d\B(R)} \cong \HH_1(\Gamma, \mathbb Z/d)$.  If $F_i$
  and $F_j$ intersect at $P_{ij}$ with local intersection multiplicity
  $\mu_{ij}$, then near the vertex $P_{ij}$, the cycle in $\Gamma$
  corresponding to $(f_i,f_j)_d$ looks like $F_i
  \xrightarrow{\mu_{ij}} P_{ij} \xrightarrow{-\mu_{ij}} F_j$.
\end{theorem}
\begin{proof}
  See \cite[Theorem~5]{F:Bgk} and \cite[\S~2]{F:Bgl}.  If moreover we
  assume each $F_i$ is a line, the proof of \cite[Theorem~4]{F:Bgk}
  shows that the Brauer group ${_\nu\B(R)}$ is generated by the set of
  symbol algebras $\{(f_i,f_j)_\nu \mid 1 \leq i < j < n\}$ over $R$
  and the only relations that arise are when three of the lines $F_i$
  meet at a common point of $\mathbb P^2$.
\end{proof}

\section{A ruled surface with an isolated singularity}
\label{sec:2}
Let $A = k[x,y]$ and $K = k(x,y)$ the quotient field of $A$.  Let $f =
f_1f_2f_3\dotsm f_n \in k[x,y]$ where $f_1, f_2, f_3, \dots, f_n$ are
linear polynomials, all of the form $f_i = a_i x + b_i y$.  Assume $n
> 2$ and $f$ is square-free in $A$.  Set $T = A[z]/(z^2 - f)$, $R =
A[f^{-1}]$ and $S = T[f^{-1}]$.  It is an exercise
\cite[Exercise~II.6.4]{H:AG} to show that $T$ is an integral domain, a
free $A$-module of rank $2$, and is integrally closed in its quotient
field, $L = K(z)$.  The surface $X=\Spec{T}$ has an isolated
singularity at the maximal ideal $(x,y,z)$. The birational equivalence
classes of such double planes have been classified in
\cite{MR1939687}.

Because $A$ is factorial, it is routine to see that $R^\ast = k^\ast
\times \langle f_1\rangle \dotsm \times \langle f_n\rangle$.  Let
$\sigma$ denote both the $R$-automorphism of $S$ and the
$K$-automorphism of $L$ defined by $z \mapsto -z$. Let $G = \{1,
\sigma\}$ be the group generated by $\sigma$ which is a group of
automorphisms of $T$, $S$ and $L$. Notice that $A =T^G$, $R = S^G$,
and $K = L^G$. Also, $S/R$ and $L/K$ are Galois by
\cite[Example~I.3.4]{M:EC}.  Because $G$ is cyclic, the image of the
homomorphism $\alpha_4$ in \eqref{eq:104} is generated by the symbol
algebras $(f_i,f)_2$ for $i = 1, \dots, n$.  We denote the image of
$\alpha_4$ by $\B^\smallsmile(S/R)$.  By \cite[Theorem~4]{F:Bgk},
${_2\B(R)} \cong (\mathbb Z/2)^{(n-1)}$ is generated by the symbol
algebras $(f_i,f_j)_2$.
Because all of the
polynomials $f_i$ are in the maximal ideal $(x,y)$, for any triple $i
< j < l$, we have the Brauer equivalence $(f_i f_j, f_l)_2 \sim (f_i,
f_j)_2$.
\begin{proposition}
  \label{prop:2.1}
  As above, let $T = k[x,y,z]/(z^2 - f)$, where $f_1, f_2, f_3, \dots,
  f_n$ are linear forms in $k[x,y]$, $n \geq 3$, and $f =
  f_1f_2f_3\dotsm f_n$ is square-free.  Then $T$ satisfies
  \begin{enumerate}[(a)]
  \item $T^\ast = k^\ast$,
  \item $\Pic(T) = 0$, and
  \item $\B(T) = 0$.
  \end{enumerate}
\end{proposition}
\begin{proof}
  If $n$ is even, we can define a grading on $k[x,y,z]$ by $\deg(x) =
  \deg(y) = 1$ and $\deg(z) = n/2$. Then $z^2 - f$ is a
  quasi-homogeneous (or weighted homogeneous) polynomial and $T$ is a
  graded ring. The subring of $T$ consisting of homogeneous elements
  of degree $0$ is $T_0 = k$ from which we get (a).  By \cite{H:Fgr},
  (b) and (c) follow.  If $n$ is odd, we define a grading on
  $k[x,y,z]$ by $\deg(x) = \deg(y) = 2$ and $\deg(z) = n$. Then $z^2 -
  f$ is a homogeneous element of degree $2n$, and $T$ is a graded ring
  with $T_0 = k$.
\end{proof}
\begin{proposition}
  \label{prop:2.2}
  In the context of Proposition~\ref{prop:2.1}, the following are
  true.
  \begin{enumerate}[(a)]
  \item If $n \geq 3$ is odd, then $\B^\smallsmile(S/R) = \B(S/R) =
    {_2\B(R)}\cong (\mathbb Z/2)^{(n-1)}$.
  \item If $n \geq 4$ is even, then $\B^\smallsmile(S/R)$ has order
    two and is generated by the Brauer class of the product of symbols
    $(f_1,f_2)_2 \dotsm (f_i,f_{i+1})_2 \dotsm (f_{n-1},f_n)_2$.
  \end{enumerate}
\end{proposition}
\begin{proof}
  The group ${_2\B(R)}$ is generated by the symbols $(f_i,
  f_j)_2$. The subgroup $\B^\smallsmile(S/R)$ is generated by all of
  the symbols $(f_i, f)_2$.  If $n$ is odd,
  \begin{equation}
    \label{eq:107}
    \begin{split}
      (f_1 f_2, f)_2 & \sim
      (f_1f_2, f_3f_4 \dotsm f_n)_2  \\
      &\sim (f_1f_2, f_3)_2 \dotsm
      (f_1f_2,  f_n)_2  \\
      & \sim (f_1, f_2)_2 \dotsm
      (f_1,  f_2)_2  \\
      & \sim (f_1, f_2)_2
    \end{split}
  \end{equation}
  from which it follows that $\B^\smallsmile(S/R)$ contains all of the
  symbols $(f_i, f_j)_2$.  Part (a) follows.  Now assume $n$ is even.
  If $i$ is odd
  \begin{equation}
    \label{eq:108}
    \begin{split}
      (f_i, f)_2 & \sim (f_i,f_1f_2)_2 \dotsm (f_i,f_if_{i+1})_2
      \dotsm
      (f_i,f_{n-1}f_n)_2 \\
      & \sim (f_1,f_2)_2 \dotsm (f_i,f_{i+1})_2 \dotsm (f_{n-1},f_n)_2
    \end{split}
  \end{equation}
  which is Brauer equivalent to
  \begin{equation}
    \label{eq:109}
    \begin{split}
      (f_i, f)_2 & \sim (f_i,f_1f_2)_2 \dotsm (f_i,f_{i-1}f_{i})_2
      \dotsm
      (f_i,f_{n-1}f_n)_2 \\
      & \sim (f_1,f_2)_2 \dotsm (f_{i-1},f_{i})_2 \dotsm
      (f_{n-1},f_n)_2
    \end{split}
  \end{equation}
  if $i$ is even.  This implies that $\B^\smallsmile(S/R)$ is a group
  of order $2$.
\end{proof}

\subsection{Quasi-homogeneous singularities}
\label{sec:2.1}
As shown in Proposition~\ref{prop:2.1}, the polynomial $h(x,y,z) = z^2
- f(x,y)$ is quasi-homogeneous and the ring $T$ is graded.  The
singularity of $X$ is a quasi-homogeneous singularity, where $X =
Z(h)$ is viewed as a hypersurface in $\mathbb A^3$.  For $t \in
k^\ast$, $h(t^2 x, t^2 y, t^n z) = t^{2n} h(x,y,z)$. Therefore, $X$ is
invariant under the $k^\ast$-action $(x,y,z) \mapsto (t^2 x, t^2 y,
t^n z)$ on $\mathbb A^3$.  Using \cite{MR0284435}, we compute the
desingularization of $X$.  There are two cases.

If $n$ is even, the singularity of $X$ can be resolved by one
blowing-up followed by a normalization. The exceptional divisor $E$ is
irreducible, isomorphic to a curve of genus $(n-2)/2$, and $E\cdot E =
-2$.

If $n$ is odd, the desingularization is accomplished by a sequence of
one blowing-up, one normalization, then $n$ more blowings-up.  The
exceptional divisor consists of $n+1$ copies of $\mathbb P^1$, call
them $E_0, E_1, \dots, E_n$. The intersection matrix is defined by
\begin{equation}
  \label{eq:110}
  \begin{split}
    E_0 \cdot E_0 & =
    -(n+1)/2\\
    E_0 \cdot E_1 & = \dotsm = E_0 \cdot E_n =
    1 \\
    E_1 \cdot E_1 & = \dotsm  = E_n \cdot E_n  = -2 \\
    E_i \cdot E_j & = 0 \text{, \quad if $0 < i$, $0 < j$, $i \neq
      j$}\ldotp
  \end{split}
\end{equation}
The determinant of $(E_i\cdot E_j)$ is equal to $2^{n-1}$.  If we
assume $k = \mathbb C$, then using \cite{MR0153682} or
\cite{MR1611536}, we compute $\HH_1(M,\mathbb Z)$, the abelianized
topological fundamental group of the intersection $M$ of $X$ with a
small sphere about the singular point $P$.  It follows that $\pi_1$ is
generated by $\gamma_0, \gamma_1, \dots, \gamma_n$, with relations $e
= \gamma_1 \dotsm \gamma_n \gamma_0^{(n+1)/2}$, $e = \gamma_0
\gamma_1^{-2}$, \dots, $e=\gamma_0 \gamma_n^{-2}$, and $\gamma_0$ is
central.  It follows that $\HH_1(M,\mathbb Z)$ is isomorphic to the
cokernel of the map defined by $(E_i\cdot E_j)$, hence is isomorphic
to $(\mathbb Z/2)^{(n-1)}$.

Notice that for $n=3$, the singularity of $X$ is rational of type
$D_4$ (see \cite[p. 258]{MR0276239} for instance).  In
\cite[(3.5)]{MR632654} it is shown that the singularity of $X =
\Spec{T}$ is rational if and only if $n = 3$.  For the remainder of
this section we distinguish between the odd $n$ and even $n$ cases.

\subsection{The case when $n$ is odd}
\label{sec:2.2}
As above, let $f = f_1f_2f_3\dotsm f_n \in k[x,y]$ where each $f_i$ is
a linear form. For this section, we assume $n$ is an odd integer
greater than or equal to $3$.  Let $R = k[x,y][f^{-1}]$, $S =
R[z]/(z^2 = f)$, and $T = k[x,y,z]/(z^2 = f)$.  An affine change of
coordinates allows us to assume $f_1 =x$ and for all $i > 1$, $f_i =
a_i x - y$.

\begin{proposition}
  \label{prop:2.3}
  In the above notation, the following are true.
  \begin{enumerate}[(a)]
  \item $S$ is a rational surface.
  \item $S$ is factorial, or in other words, $\Pic{S} = 0$.
  \item $S^\ast = k^\ast \times \langle z\rangle \times \langle f_2
    \rangle \times \dotsm \times \langle f_n \rangle$
      \item
    \[
    \HH^i(G,S^\ast) \cong
    \begin{cases}
      R^\ast & \text{if $i = 0$}\\
      \frac{\langle f_2\rangle}{\langle f_2^2\rangle} \times \dotsm
      \times \frac{\langle f_n\rangle}{\langle f_n^2\rangle} \cong
      (\mathbb Z/2)^{(n-1)}
      & \text{if $i = 2, 4, 6, \dots$}\\
      \langle 1\rangle & \text{if $i=1, 3, 5, \dots $}
    \end{cases}
    \]
  \end{enumerate}
\end{proposition}
\begin{proof}
  Consider the two $k$-algebra isomorphisms
  \begin{equation}
    \label{eq:111}
    \begin{CD}
      S = \frac{k[x,y,z][z^{-1}]}{z^2 = f(x,y)} @>{\alpha}>>
      \frac{k[x,v,w][w^{-1}]}{w^2 = x f(1,v)} @>{\beta}>>
      k[v,w][w^{-1}, f(1,v)^{-1}]
    \end{CD}
  \end{equation}
  where $\alpha$ is defined by $x \mapsto x$, $y \mapsto xv$, $z
  \mapsto x^{(n-1)/2} w$, and $\beta$ is defined by $x \mapsto w^{2}
  f(1,v)^{-1} $, $v \mapsto v$, $w \mapsto w$.  Let $U$ denote the
  ring on the right hand side of \eqref{eq:111}. Since $U$ is rational
  and factorial, this proves (a) and (b).  The group of units in $U$ is equal
  to $k^\ast \times \langle w\rangle \times \langle f_2(1,v) \rangle
  \times \dotsm \times \langle f_n(1,v) \rangle$.  Using this, one
  checks that the group of units in $S$ is generated by $k^\ast$ and
  the elements $z, f_2, \dots, f_n$.  Apply
  \cite[Theorem~10.35]{R:IHA} to get part (d).
\end{proof}

\begin{proposition}
  \label{prop:2.4}
  In the context of Section~\ref{sec:2.2}, the following are
  true.
  \begin{enumerate}[(a)]
  \item $\Cl(T) \cong (\mathbb Z/2)^{(n-1)}$ is generated by the prime
    divisors $I_i = (z, f_i)$, where $i = 2, \dots, n$.
  \item
    \[
    \HH^i(G, \Cl(T)) \cong
    \begin{cases}
      \Cl(T) \otimes \mathbb Z/2 \cong (\mathbb Z/2)^{(n-1)}
      & \text{if $i$ is odd}\\
      \Cl(T)^G = {\Cl(T)} \cong (\mathbb Z/2)^{(n-1)} & \text{if $i$ is
        even}
    \end{cases}
    \]
  \end{enumerate}
\end{proposition}
\begin{proof}
  In $T$, the minimal primes of $z$ are $I_1 = (f_1,z)$, \dots, $I_n =
  (f_n,z)$.  A local parameter for $T_{I_i}$ is $z$. The divisor of
  $z$ on $\Spec{T}$ is $\Div(z) = I_1+\dots+I_n$. The only minimal
  prime of $f_i$ is $I_i$. The divisor of $f_i$ is $2 I_i$.  Since $S
  = T[z^{-1}]$, by Nagata's Theorem \cite[Theorem~7.1]{F:DCG},
  $\Cl(T)$ is generated by $I_1, \dots, I_n$, subject to the relations
  $I_1 + \dots + I_n \sim 0$, $2I_1 \sim 0$, \dots, $2I_n \sim 0$.
  Therefore, $\Cl(T) \cong (\mathbb Z/2)^{(n-1)}$ is generated by any
  $n-1$ of $I_1, \dots, I_n$ and this proves (a).  We remark that (a)
  also follows from \cite{MR2389245} or \cite{MR751470}.  Since the
  group $G = \langle \sigma \rangle$ acts trivially on $\Cl(T)$, we
  get (b).
\end{proof}
\begin{proposition}
  \label{prop:2.5}
  In the context of Proposition~\ref{prop:2.3}, the sequence
  \[
  0 \rightarrow \B(S/R) \rightarrow \B(R) \rightarrow \B(S)
  \rightarrow 0
  \]
  is exact.
\end{proposition}
\begin{proof}
  Let $U$ be as in the proof of Proposition~\ref{prop:2.3}.  By
  \cite[Theorem~4]{F:Bgk}, the Brauer group ${_m\B(U)}$ is isomorphic
  to $(\mathbb Z/m)^{(n-1)}$. It is generated by the symbol algebras
  $(w,f_i(1,v))_m$.
  The Brauer group ${_m\B(R)}$ is isomorphic to
  $(\mathbb Z/m)^{(n-1)}$ and is generated by the symbol algebras
  $(x,f_i(x,y))_m$.  Under the isomorphism $\beta\alpha$ in
  \eqref{eq:111}, the algebra $(x,f_i(x,y))_m$ maps to
  \[
  (w^2 f_i(1,v)^{-1}, f_i(1,v))_m \sim (w^2, f_i(1,v))_m \sim (w,
  f_i(1,v))^2_m\ldotp
  \]
  As a homomorphism of abstract groups, $\B(R) \rightarrow \B(S)$ can
  be viewed as ``multiplication by $2$''.
\end{proof}
\begin{remark}
  \label{remark:2.6}
  It follows from Proposition~\ref{prop:2.3}(a) that in \eqref{eq:104},
  the sequence of Chase, Harrison, Rosenberg, the groups involving
  $\Pic{S}$ are trivial. The map $\alpha_4$ is an isomorphism.
\end{remark}
\begin{remark}
\label{remark:2.7}
  Propositions~\ref{prop:2.2} and \ref{prop:2.4} imply that $\B(S/R)
  \cong \HH^1(G,\Cl(T))$.
\end{remark}

In Proposition~\ref{prop:2.8}, let $X = \Spec{T}$ and $P$ the
singular point of $X$. By a desingularization of $P$ we mean a proper,
surjective, birational morphism $\pi: \tilde X \rightarrow X$ such
that $\tilde X$ is nonsingular and $\pi$ is an isomorphism on $\tilde
X -E$, where $E = \pi^{-1}(P)$.  Let $T^h = \mathcal O^h_P$ be the
henselization of $\mathcal O_{P}$ and $\hat{T}=\hat{\mathcal O}_P$ the
completion.

\begin{proposition}
  \label{prop:2.8}
   The following are true for any desingularization $\pi : \tilde X
   \rightarrow X$.
   \begin{enumerate}[(a)]
   \item $\Cl(X) = \Cl(\mathcal O_{P})$
   \item $\Cl(\mathcal O^h_P) = \Cl(\hat{\mathcal O}_P)$ is isomorphic
     to $\Cl(X) \oplus V$, where $V$ is a finite dimensional
     $k$-vector space.
   \item If $n=3$, the singularity $P$ is rational and $0=\B(X) =
     \HH^2(X,\mathbb G_m)$.
   \item If $n>3$, the singularity $P$ is irrational and
     $\HH^2(X,\mathbb G_m)$ is isomorphic to the $k$-vector space $V$ of
     Part~(b).
   \item $0= \B(\tilde X) = \HH^2(\tilde X,\mathbb G_m)$.
   \item $0 = \B(X-P)$.
   \end{enumerate}
 \end{proposition}
 \begin{proof}
   By \cite[Corollary~10.3]{F:DCG}, $\Cl(T) = \Cl(\mathcal O_P)$.  By
   Mori's Theorem \cite[Corollary~6.12]{F:DCG}, the natural maps
   $\Cl(\mathcal O_{P}) \rightarrow \Cl(\mathcal O^h_P) \rightarrow
   \Cl(\hat{\mathcal O}_P)$ are one-to-one.  By Artin Approximation
   \cite{MR0268188}, $\Cl(\mathcal O^h_P) = \Cl(\hat{\mathcal O}_P)$.
   Assume for now that $\pi: \tilde X \rightarrow X$ is a minimal
   resolution, so the exceptional divisor is as described in
   \eqref{eq:110}.  The Neron-Severi group $\mathbf E =
   \bigoplus_{i=0}^n \mathbb Z E_i$ is the free abelian group on the
   irreducible components of the exceptional fiber.  The cokernel of
   the intersection matrix \eqref{eq:110} applied to $\mathbf E$ is
   called $\HH(T)$.  Then $\HH(T)$ is isomorphic to $(\mathbb
   Z/2)^{(n-1)}$.  Given a divisor $D$ on $\tilde X$, sending the
   class of $D$ in $\Pic(\tilde X)$ to $\sum (D\cdot E_i) E_i$ defines
   a homomorphism $\theta : \Pic(\tilde X) \rightarrow \mathbf E$. The
   kernel of $\theta$ is denoted $\Pic^0(T)$.  The cokernel of
   $\theta$ is denoted $\GG(T)$.  By
   \cite[Proposition~15.3]{MR0276239}, the groups $\HH(T)$,
   $\Pic^0(T)$, $\GG(T)$ depend only on $T$ and not $\tilde X$.  By
   \cite[Proposition~16.3]{MR0276239}, there is a commutative diagram
   \begin{equation}
     \label{eq:112}
     \begin{CD}
       0 @>>> \Pic^0(T) @>>> \Cl(T) @>>> \HH(T) @>>> \GG(T) @>>> 0 \\
       @. @VV{\alpha}V @VV{\beta}V @VV{=}V @VVV \\
       0 @>>> \Pic^0(\hat T) @>>> \Cl(\hat T) @>>> \HH(\hat T) @>>> 0
     \end{CD}
   \end{equation}
   with exact rows.  Let $F_i$ denote the divisor on $X$ corresponding
   to the prime ideal $I_i = (z,f_i)$.  If $\tilde F_i$ is the strict
   transform of $F_i$ on $\tilde X$, then $\tilde F_i \cdot E_j =
   \delta_{ij}$, the Kronecker delta function.  This shows that
   $\GG(T) = 0$. By Proposition~\ref{prop:2.4}(a), $\Cl(T) \cong
   \HH(T)$ and $\Pic^0(T)=0$. From \eqref{eq:112}, $\Pic^0(\hat T)
   \cong \Cl(\hat T)/\Cl(T)$.  It follows from Artin's construction
   in \cite[pp. 486--488]{MR0146182} that the group $\Pic^0(\hat T)$
   is a finite dimensional $k$-vector space.  For an exposition, see
   \cite[pp. 423--426]{MR0379489}, especially the proof of
   Corollary~4.5 and the Remark which follows.  This proves (b).
   The rest follows from \cite{F:Bgd}.
 \end{proof}
\begin{question}
  \label{question:2.9}
  As in Proposition~\ref{prop:2.3}, let $f = f_1f_2f_3\dotsm f_n \in
  k[x,y]$ where $f_1, f_2, f_3, \dots, f_n$ are linear forms, and $n$
  is an odd integer greater than $3$.  Assume $f_i \neq x$ and $f_i
  \neq y$ for all $i$.  Let $K = k(x,y)$, $L = K(\sqrt{f})$. Consider
  $L_1 = L(\sqrt{x})$, $L_2 = L(\sqrt{y})$, $L_3 = L(\sqrt{x},
  \sqrt{y})$. Then $L_3/L$ is an abelian Galois extension of degree
  $2^2$. We also view $L_3$ as the function field of the surface $S_3$
  which is defined by $z^2 = f(x^2, y^2)$. Can we use the extension
  $S_3/S$ to reduce to the case where $n$ is even?
\end{question}

\subsection{The case when $n$ is even}
\label{sec:2.3}
As above, let $T = k[x,y,z]/(z^2 - f)$, where $f_1, f_2, f_3, \dots,
f_n$ are linear forms in $k[x,y]$, $f = f_1f_2f_3\dotsm f_n$ is
square-free, and $n \geq 4$ is even.  Without loss of generality
assume $f_1 = x$ and $f_2 = y$.  Write $\kappa$ for $n / 2$.  As
mentioned above, the singularity of $T$ is resolved by blowing up the
maximal ideal $(x,y,z)$ and then taking the integral closure.  The
homomorphism of $k$-algebras
\begin{equation}
  \label{eq:113}
  T =
  \frac{k[x,y,z]}{(z^2 - x y  f_3f_4\dotsm f_n)}
  \rightarrow
  \frac{k[x,v,w]}{(w^2 - v(a_3-b_3v) (a_4-b_4v)
    \dotsm  (a_n-b_nv) )} = \tilde T
\end{equation}
defined by $x\mapsto x$, $y\mapsto xv$, $z\mapsto x^\kappa w$ is an
affine piece of the resolution of $T$.  Let $C$ be the affine
hyperelliptic curve defined by $w^2 - f(1,v)$.  It is clear that
$\tilde T = k[x] \otimes_k \mathcal O(C)$, hence is regular.  Both $T$
and $\tilde T$ are integral domains.  The map \eqref{eq:113} is
birational, because we identify $v = y x^{-1}$ and $w = z x^{-\kappa}$
with elements of the quotient field of $T$.  So \eqref{eq:113} is
one-to-one.
\begin{proposition}
  \label{prop:2.10}
  If $n$ is even, and $S$ is as in the top of
  Section~\ref{sec:2}, then $S$ is birational to a ruled surface
  $\mathbb A^1 \times C$, where $C$ is an affine hyperelliptic curve
  defined by an equation of the form $y^2 = (x-\lambda_2)
  (x-\lambda_3) \dotsm (x-\lambda_n)$.  The genus of the nonsingular
  completion of $C$ is $(n-2)/2$.
\end{proposition}
\begin{proof}
  The previous paragraph proves that $T$ is birational to a ruled
  surface as claimed. But $S = T[z^{-1}]$ is birational to $T$.  By
Section~\ref{sec:3.2}, if $\bar C$ is the complete nonsingular
  model for $C$, then the genus of $\bar C$ is $\kappa -1 = (n-2)/2$.
\end{proof}

The ideal of $\tilde T$ lying over $(x,y,z)$ is the principal ideal
$(x)$.  Upon tensoring with $T[x^{-1}]$, \eqref{eq:113} becomes an
isomorphism.  The only minimal prime of $x$ in $T$ is $I = (x,z)$.
Upon localizing at $I$, $z$ is a local parameter for $T_I$.  The
divisor of $x$ on $\Spec{T}$ is $2I$.  By \cite[Theorem~8.1]{F:DCG},
$\Cl(\tilde T) = \Cl(C)$.  Because $x$ is irreducible in $\tilde T$,
it follows that $\Cl(T[x^{-1}]) = \Cl(\tilde T[x^{-1}]) = \Cl(\tilde
T) = \Cl(C)$.  The group $\mathcal O^\ast(C)$ of units on $C$ is equal
to $k^\ast$, because for a nonsingular completion $\bar C$ of $C$
there is only one point at infinity $\bar C - C$.  It follows that the
group of units of $T[x^{-1}]$ is $k^\ast \times \langle x \rangle$.
Nagata's Theorem \cite[Theorem~7.1]{F:DCG} becomes
\begin{equation}
  \label{eq:114}
  \begin{CD}
    0 @>>> \mathbb ZI/ 2\mathbb Z I @>>> \Cl(T) @>{\sigma}>>
    \Cl(T[x^{-1}]) @>>> 0\ldotp
  \end{CD}
\end{equation}

Let $T^h$ denote the henselization of $T$ at the maximal ideal
$(x,y,z)$.  Let $X = \Spec{T}$ and $P$ the singular point of $X$.  Let
$T^h = \mathcal O^h_P$ be the henselization of $\mathcal O_{P}$ and
$\hat{T}=\hat{\mathcal O}_P$ the completion.
 \begin{proposition}
   \label{prop:2.11}
   If $\pi: \tilde X \rightarrow X$ is any desingularization of $P$ and
 $E = \pi^{-1}(P)$,
   then the following are true.
   \begin{enumerate}[(a)]
   \item $\Cl(X) = \Cl(\mathcal O_{P})$.
   \item $\Cl(T^h) = \Cl(\hat{T})$ is isomorphic
     to $\Cl(X) \oplus V$, where $V$ is a finite dimensional
     $k$-vector space.
   \item $\HH^2(X,\mathbb G_m)$  is isomorphic to the $k$-vector space  $V$ of
     Part~(b).
   \item $0= \B(\tilde T) = \HH^2(\tilde T,\mathbb G_m)$.
   \item $0= \B(\tilde X) = \HH^2(\tilde X,\mathbb G_m)$.
   \item $\B(X-P)$ is isomorphic to $\HH^1(E,\mathbb Q/\mathbb Z)$,
     which is isomorphic to $(\mathbb Q/\mathbb Z)^{(n-2)}$.
   \end{enumerate}
 \end{proposition}
 \begin{proof}
   Parts (a) and (b) are proved as in Proposition~\ref{prop:2.8}.  By
   \cite[Corollaire~(1.2)]{G:GBIII}, $\B(C) = 0$. Since $\tilde T$ is
   obtained from $\mathcal O(C)$ by adjoining an indeterminate,
   $\B(\tilde T) = 0$.  Since $\tilde T$ is regular, $\B(\tilde T) =
   \HH^2(\tilde T,\mathbb G_m)$ by \cite{H:WBr}.  This proves (d).
   The rest follows from \cite{F:Bgd}.
\end{proof}


For the Galois extension $R \rightarrow S$, Proposition~\ref{prop:2.12}
computes the terms in \eqref{eq:104}, the sequence of Chase, Harrison and Rosenberg.

\begin{proposition}
  \label{prop:2.12}
  Let $f$, $T$ and $S$ be as above.  The following are true.
  \begin{enumerate}[(a)]
  \item $\B^\smallsmile(S/R)= \B(S/R)$ has order two.
  \item $\Pic(S)$ is divisible and
    \[
    \HH^i(G,\Pic(S)) \cong
    \begin{cases}
      \Pic(S) \otimes \mathbb Z/2 = 0
      & \text{if $i$ is odd}\\
      \Pic(S)^G = {_2\Pic(S)} \cong (\mathbb Z/2)^{(n-2)} & \text{if
        $i$ is even}
    \end{cases}
    \]
  \item $S^\ast = k^\ast \times \langle z\rangle \times \langle f_2
    \rangle \times \dotsm \times \langle f_n \rangle$
  \item
    \[
    \HH^i(G,S^\ast) \cong
    \begin{cases}
      R^\ast & \text{if $i = 0$}\\
      \frac{\langle f_2\rangle}{\langle f_2^2\rangle} \times \dotsm
      \times \frac{\langle f_n\rangle}{\langle f_n^2\rangle} \cong
      (\mathbb Z/2)^{(n-1)}
      & \text{if $i = 2, 4, 6, \dots$}\\
      \langle 1 \rangle & \text{if $i=1, 3, 5, \dots $}
    \end{cases}
    \]
  \end{enumerate}
\end{proposition}
\begin{proof}
  Let $\tilde T$ denote the ring on the right hand side of
  \eqref{eq:113}, and $U = \tilde T[x^{-1}]$ the localized ring.  Then
  $U = k[x,x^{-1}] \otimes \mathcal O(C)$, where $C$ is the affine
  hyperelliptic curve $w^2 = v(a_3-b_3v) (a_4-b_4v) \dotsm
  (a_n-b_nv)$.
  We know that $\Cl(C) = \Cl(\mathbb A^1
  \times C)$.  It follows that $\Cl(U) = \Cl(C)$.
  Proposition~\ref{prop:3.1} says that
  \begin{equation}
    \label{eq:115}
    \HH^i\left(G, \Cl(U)\right) =
    \begin{cases}
      0 & \text{if $i$ is odd}\\
      {_2\Cl(C)} \cong (\mathbb Z/2)^{(n-2)} & \text{if $i$ is even.}
    \end{cases}
  \end{equation}
  By Nagata's Theorem \cite[Theorem~7.1]{F:DCG}, $\beta : \Cl(U)
  \rightarrow \Cl(U[w^{-1}])$ is onto and the kernel of $\beta$ is
  generated by those prime divisors containing $w$.  In $U$ there are
  $n-1$ minimal primes of $w$. They are $\{I_i = (w,
  f_i(1,v))\}_{i=2}^n$.  The only minimal prime of $f_i(1,v)$ is
  $I_i$.  Upon localizing $U$ at $I_i$ we see that $w$ is a local
  parameter and the valuation of $f_i(1,v)$ is $2$. The divisors of
  $w$ and $f_i(1,v)$ are $\Div(w) = I_2+\dots+I_n$, $\Div(f_i(1,v)) =
  2I_i$.  The sequence
  \begin{equation}
    \label{eq:116}
    1 \rightarrow  U^\ast \rightarrow U[w^{-1}]^\ast \xrightarrow{\Div}
    \bigoplus_{i=2}^n \mathbb Z I_i \xrightarrow{\theta}
    \Cl(U) \xrightarrow{\beta}
    \Cl(U[w^{-1}]) \rightarrow 0
  \end{equation}
  is exact.  By Section~\ref{sec:3.2}, the kernel of $\beta$ is
  the subgroup ${_2\Cl(U)}$, an elementary $2$-group of rank $n-2$.
  From this we see that the sequence
  \begin{equation}
    \label{eq:117}
    1 \rightarrow  U^\ast \rightarrow  U[w^{-1}]^\ast \rightarrow
    \langle w \rangle
    \times   \langle f_3(1,v)  \rangle  \times \dotsm
    \times   \langle f_n(1,v)  \rangle
  \end{equation}
  is split-exact.  The ring $S$ is isomorphic to the localization
  $U[w^{-1}]$, so part (c) follows from \eqref{eq:117}. Apply
  \cite[Theorem~10.35]{R:IHA} to get part (d). The sequence
  \begin{equation}
    \label{eq:118}
    0 \rightarrow (\mathbb Z/2)^{(n-2)} \rightarrow
    \Cl(U) \xrightarrow{\beta}
    \Cl(U[w^{-1}])
    \rightarrow 0
  \end{equation}
  is exact.  Since $\Cl(C)$ is divisible,
  $\Cl(U)$ is divisible, so $\beta$ can be viewed as ``multiplication
  by $2$''.  On ${_2\Cl(U)}$ the group $G$ acts trivially.  The long
  exact sequence of cohomology associated to \eqref{eq:118} is
  \begin{equation}
    \label{eq:119}
    \dotsm 
    \xrightarrow{\partial^{i}}
    (\mathbb Z/2)^{(n-2)}
    \rightarrow
    \HH^i(G,   \Cl(U))
    \xrightarrow{\beta^i}
    \HH^i(G,   \Cl(U[w^{-1}]))
    \xrightarrow{\partial^{i+1}}
    (\mathbb Z/2)^{(n-2)}
    \rightarrow
    \dotsm
  \end{equation}
  where $\beta^i$ is the zero map.  Part (b) follows from
  \eqref{eq:115} and \eqref{eq:119}.  Proposition~\ref{prop:2.2} says
  $\B^\smallsmile(S/R)$ is a cyclic group of order two.  Together with
  sequence \eqref{eq:104}, part (b), and part (d), this proves part
  (a).
\end{proof}
\begin{proposition}
\label{prop:2.13}
  In the context of Proposition~\ref{prop:2.12}, the
  following are true.
  \begin{enumerate}[(a)]
  \item There is an exact sequence
    \[
    0 \rightarrow \mathbb Z/2 \rightarrow {_2\B(R)} \rightarrow
    {_2\B(T[f_1^{-1}])} \rightarrow 0
    \]
  \item As an abstract group, $\B(S)$ is isomorphic to
    $\HH^1(C,\mathbb Q/\mathbb Z) \oplus (\mathbb Q/\mathbb
    Z)^{(n-1)}$, which is isomorphic to $(\mathbb Q/\mathbb Z)^{(n-2)}
    \oplus (\mathbb Q/\mathbb Z)^{(n-1)}$.
  \item The cokernel of $\B(R) \rightarrow \B(S)$ is isomorphic to
    $\HH^1(C,\mathbb Q/\mathbb Z)$.
  \end{enumerate}
\end{proposition}
  \begin{proof}
  By \cite[Corollary~1.4]{F:Bgl}, ${\B(U)} \cong \HH^1(C,\mathbb
  Q/\mathbb Z)$. It follows that ${_2\B(U)}$ is isomorphic to
  $\HH^1(C,\mathbb Z/2)
  \cong (\mathbb Z/2)^{(2\kappa -2)}$.  Let $L$ denote the quotient
  field of $S$. We view $S$, $U$ and $\tilde T$ as subrings of $L$,
  all birational to each other.

  The Brauer group of $R$ is generated by the symbol algebras
  $(x,f_i)_m$. To determine the image of $\B(R) \rightarrow \B(S)$, we
  first consider the symbol $\Lambda = (x,v)_m$. Assume that we have
  extended the scalars so that $\Lambda$ is a central simple
  $L$-algebra.  Since $\B(\tilde T)=0$, we use the exact sequence
  \eqref{eq:105} to measure the ramification of $\Lambda$ at the prime
  divisors of $\tilde T$.  We need only consider primes in the divisor
  $(x) + (v)$ on $\Spec{\tilde T}$.  Any prime of $\tilde T$
  containing $v$ must also contain $w$, so let $I = (v,w)$.  The
  element $x \in \tilde T$ is irreducible and $J = (x)$ is a prime
  ideal.  The ramification map $a$ agrees with the tame symbol
  \eqref{eq:106} at the divisors $I$ and $J$.

  To compute the ramification $a_I(\Lambda)$ at the prime $I$, we
  notice that $w$ is a local parameter in the local ring $\tilde T_I$.
  The valuation of $v$ is $2$, the valuation of $x$ is $0$. The tame
  symbol becomes $x^2$. The residue field at $I$ is the quotient field
  of $\tilde T/I$, which we identify with $k(x)$.  Therefore
  \begin{equation}
    \label{eq:120}
    a_I(\Lambda) = k(x)(x^{2/m})\ldotp
  \end{equation}
  In $\tilde T/I$ the element $x$ belongs to only one prime, so the
  ramification map $r$ is
  \begin{equation}
    \label{eq:121}
    r_p\bigl(a_I(\Lambda)\bigr) =
    \begin{cases}
      2 & \text{if $p = (x)$} \\
      0 & \text{otherwise.}
    \end{cases}
  \end{equation}
  Therefore \eqref{eq:120} represents an element of $\HH^1(k(x),\mathbb
  Z/m)$ of order $m/ \gcd(m,2)$.
  
  Modulo $J = (x)$, $\tilde T$ is isomorphic to $k[v,w]/(w^2
  -v(a_3-b_3v)(a_4-b_4v) \dotsm (a_n-b_nv))$, the affine coordinate
  ring of the non-singular curve $C$. By $C$ we also denote the
  corresponding divisor on $\Spec{\tilde T}$, and by $K(C)$ the
  function field of $C$.  In $\mathcal O_C$, $x$ is a local parameter
  and has valuation $1$. The valuation of $v$ is $0$. The tame symbol
  \eqref{eq:106} becomes $v^{-1}$.  Therefore
  \begin{equation}
    \label{eq:122}
    a_J(\Lambda) =
    K(C)(v^{-1/m})\ldotp
  \end{equation}
  To compute $r\bigl(K(C)(v^{-1/m})\bigr)$, notice that the divisor of
  $v$ on $C$ is $2 P_2$, where $P_2$ is the prime divisor defined by
  $v = w = 0$.  Therefore,
  \begin{equation}
    \label{eq:123}
    r_p\bigl(a_J(\Lambda)\bigr) =
    \begin{cases}
      -2 & \text{if $p = P_2$} \\
      0 & \text{otherwise.}
    \end{cases}
  \end{equation}
  For all $m$, \eqref{eq:122} represents an element of order $m$ in
  $\HH^1(K(C),\mathbb Z/m)$.  Notice that for $m = 2$, the quadratic
  extension $a_J(\Lambda)$ corresponds to the element of ${_2\Pic(C)}$
  generated by the prime divisor $P_2$.

  To prove (a), assume $m= 2$.  As in the proof of
  Proposition~\ref{prop:2.12}, let $U = \tilde T[x^{-1}]$.  By the
  computations above we see that $\Lambda$ ramifies only along the
  curve $C$. Therefore, by sequence \eqref{eq:105} applied to $U$, we
  see that the image of $\Lambda$ in $\B(L)$ lands in the image of
  $\B(U) \rightarrow \B(L)$.  By
  Section~\ref{sec:3.2}, the group
  ${_2\Pic(C)}$ is generated by the primes $(v,w)$, $(a_3-b_3 v,w)$,
  \dots, $(a_n-b_n v,w)$.  Consider the subgroup $H$ of order
  $2^{n-2}$ in ${_2\B(R)}$ generated by the symbols $(x,y)_2$,
  $(x,f_3)_2$, \dots, $(x,f_n)_2$.  Iterating the previous argument
  for $\Lambda_3= (x,f_3)_2$, \dots, $\Lambda_n= (x,f_n)_2$, shows
  that the image of $H$ under $\B(R) \rightarrow \B(L)$ is equal to
  the image of ${_2\B(U)} \rightarrow \B(L)$.

  Since $S$ is isomorphic to the localization $\tilde T[x^{-1},
  w^{-1}]$, we view $\Spec{S}$ as an open in $\Spec{\tilde T} =
  \mathbb A^1 \times C$.  The closed complement $\Spec{\tilde T} -
  \Spec{S}$ is $C + L_2 + \dots + L_{n}$, each $L_i$ being a copy of
  $\mathbb A^1$. The intersection numbers are $L_i \ldotp L_j = 0$,
  $L_i \ldotp C = 1$. The formula in part (b) for $\B(S)$ follows from
  \cite[Corollary~1.6]{F:Bgl}.  By \eqref{eq:120} and \eqref{eq:121},
  the composite
  \[
  \B(R) \rightarrow \B(S) \rightarrow \bigoplus_{i=2}^n\HH^1(\mathcal
  O(L_i)[x^{-1}], \mathbb Q/\mathbb Z)
  \]
  is onto.  By \eqref{eq:122} and \eqref{eq:123}, the cokernel of $\B(R)
  \rightarrow \B(S)$ is isomorphic to $\HH^1(C, \mathbb Q/\mathbb Z)$
  modulo the subgroup of elements annihilated by $2$, which is (c).
\end{proof}
\begin{proposition}
  \label{prop:2.14}
  Assume we are in the context of Proposition~\ref{prop:2.12}. Then
  ${_2\Cl(T)}$ is isomorphic to $(\mathbb Z/2)^{(n-1)}$, is
  generated by the divisor classes of the prime ideals $(f_1,z)$,
  \dots, $(f_{n-1},z)$, and
  \[
  \HH^i(G,\Cl(T)) \cong
  \begin{cases}
    \Cl(T)\otimes \mathbb Z/2 \cong \mathbb Z/2
    & \text{if $i$ is odd}\\
    \Cl(T)^G = {_2\Cl(T)} \cong (\mathbb Z/2)^{(n-1)} & \text{if $i$
      is even.}
  \end{cases}
  \]
\end{proposition}
\begin{proof}
  In $T$ the only height one prime containing $f_j$ is $I_j =
  (f_j,z)$. A local parameter for the discrete valuation ring
  $T_{I_j}$ is $z$. One checks that the divisor of $f_j$ is $2 I_j$.
  Using \eqref{eq:113}, we see that $T[f_1^{-1}]$ is isomorphic to
  $k[x,x^{-1}] \otimes \mathcal O(C)$, where $C$ is the affine
  hyperelliptic curve $w^2 = v(a_3 - b_3v) \dotsm (a_n - b_n v)$.  The
  group $\mathcal O^\ast(C)$ of units on $C$ is equal to $k^\ast$,
  because for a nonsingular completion $\bar C$ of $C$ there is only
  one point at infinity $\bar C - C$.
  It follows that the group of
  units of $T[f_1^{-1}]$ is $k^\ast \times \langle f_1 \rangle$ and
  the class group of $T[f_1^{-1}]$ is isomorphic to the class group of
  $C$.  The groups $\HH^i(G,\Cl(T[f_1^{-1}]))$ are as in
  \eqref{eq:115}.  The genus of $C$ is $\kappa-1$, where $\kappa =
  n/2$.  By Section~\ref{sec:3.2}, the group ${_2\Cl(C)}$, has
  order $2^{n-2}$ and the divisors $(v,w), (a_3-b_3v,w), \dots,
  (a_{n-1}-b_{n-1}v,w)$ are a $\mathbb Z/2$-basis for ${_2\Cl(C)}$.
  Nagata's sequence \cite[Theorem~7.1]{F:DCG} becomes
  \begin{equation}
    \label{eq:124}
    0\rightarrow \mathbb Z/2 \xrightarrow{\alpha} \Cl(T)
    \xrightarrow{\beta} \Cl\left(T[f_1^{-1}]\right) \rightarrow 0
  \end{equation}
  where $\alpha$ maps $1$ to the divisor class generated by $I_1$.
  Under $\beta$, we have the assignments of divisor classes $I_2
  \mapsto (v,w)$, $I_3 \mapsto (a_3-b_3 v,w)$, \dots, $I_n \mapsto
  (a_n-b_n v,w)$.  We see that the divisors $I_1, \dots, I_n$ generate
  a subgroup of $\Cl{(T)}$ of order $2^{n-1}$. The diagram of abelian
  groups
  \begin{equation}
    \label{eq:125}
    \begin{CD}
      0 @>>> \mathbb Z/2 @>{\alpha}>> \Cl(T) @>{\beta}>>
      \Cl\left(T[f_1^{-1}]\right) @>>> 0\\
      @.  @VV{2}V @VV{2}V
      @VV{2}V\\
      0 @>>> \mathbb Z/2 @>{\alpha}>> \Cl(T) @>{\beta}>>
      \Cl\left(T[f_1^{-1}]\right) @>>> 0
    \end{CD}
  \end{equation}
  is commutative, where the vertical maps are ``multiplication by
  $2$''. The group $\Cl\left(T[f_1^{-1}]\right)$ is divisible.  The
  Snake Lemma and the previous computations imply that ${_2\Cl(T)}
  \cong (\mathbb Z/2)^{(n-1)}$ and $\Cl(T)\otimes \mathbb Z/2 \cong
  \mathbb Z/2$.  A $\mathbb Z/2$-basis for ${_2\Cl(T)}$ is the set of
  divisor classes of $I_1$, \dots, $I_{n-1}$.  The long exact sequence
  of cohomology associated to \eqref{eq:124} is
  \begin{multline}
    \label{eq:126}
    0 \rightarrow \mathbb Z/2 \xrightarrow{\alpha} \Cl(T)^G
    \xrightarrow{\beta} \Cl\left(T[f_1^{-1}]\right)^G \\
    \xrightarrow{\partial^1} \mathbb Z/2 \xrightarrow{\alpha^1}
    \HH^1(G, \Cl(T)) \xrightarrow{\beta^1} \HH^1\left(G,
      \Cl\left(T[f_1^{-1}]\right)\right) \\
    \xrightarrow{\partial^2} \mathbb Z/2 \xrightarrow{\alpha^2}
    \HH^2(G, \Cl(T)) \xrightarrow{\beta^2}
    \HH^2\left(G, \Cl\left(T[f_1^{-1}]\right)\right) \\
    \xrightarrow{\partial^3} \mathbb Z/2 \xrightarrow{\alpha^3}
    \HH^3(G, \Cl(T)) \xrightarrow{\beta^3} \HH^3\left(G,
      \Cl\left(T[f_1^{-1}]\right)\right) \xrightarrow{\partial^4}
    \dotsm
  \end{multline}
  The terms with $T[f_1^{-1}]$ come from \eqref{eq:115}.  Since $\beta
  : {_2\Cl(T)} \rightarrow {_2\Cl\left(T[f_1^{-1}]\right)}$ is onto,
  we see that $\Cl(T)^G = {_2\Cl(T)}$ and $\alpha^1$ is an
  isomorphism.  Periodicity implies $\alpha^3$ is an isomorphism.
\end{proof}
\begin{remark}
  \label{remark:2.15}
  Propositions~\ref{prop:2.2} and \ref{prop:2.14} imply that $\B(S/R)
  \cong \HH^1(G,\Cl(T))$.
\end{remark}
\begin{proposition}
  \label{prop:2.16}
  In the context of Proposition~\ref{prop:2.12}, let $X = \Spec{T}$ and
  let $P$ be the singular point of $X$.  The sequence
  \[
  0 \rightarrow \mathbb Z/2 \rightarrow {_2\B(R)} \rightarrow
  {_2\B(X-P)} \rightarrow 0
  \]
  is exact, where ${_2\B(X-P)}$ is viewed as a subset of $\B(S)$.
\end{proposition}
\begin{proof}
  Let $X_i$ be the open in $X$ defined by $f_i \neq 0$. By
  Proposition~\ref{prop:2.13}(a),
  \[
  0 \rightarrow \mathbb Z/2 \rightarrow {_2\B(R)} \rightarrow
  {_2\B(X_i)} \rightarrow 0
  \]
  is exact for $i = 1, 2$. Therefore, as a subsets of $\B(S)$,
  ${_2\B(X_1)} = {_2\B(X_2)}$. The ideal $(f_1, f_2)$ is primary for
  the maximal ideal $(x,y, z)$. Therefore, $X_1 \cup X_2 = X - P$. The
  Mayer-Vietoris sequence
  \[
  0 \rightarrow \B(X_1 \cup X_2) \rightarrow \B(X_1) \oplus \B(X_2)
  \rightarrow \B(X_1 \cap X_2)
  \]
  is exact \cite[Exercise~III.2.24]{M:EC}.
\end{proof}
 
\section{An affine double plane ramified along a hyperelliptic curve}
\label{sec:4}

Let $f = y^2 - p(x)$, where $p(x) \in k[x]$ is monic, and has degree
$d \geq 2$.  Let $A =k[x,y]$, $T = A[z]/(z^2-f)$, $R = A[f^{-1}]$, $S
= T[z^{-1}]$.  Let $p(x) = \ell_1^{e_1} \dotsm \ell_v^{e_v}$ be the
unique factorization of $p(x)$. Assume $\ell_1, \dots, \ell_v$ are
distinct and of the form $\ell_i = x-\alpha_i$.  Let $D = \gcd(e_1,
\dots, e_v)$.  Let $F = Z(y^2 - p(x))$ be the subset of $\mathbb A^2$.
Then $F$ is nonsingular if and only if each $e_i = 1$.  The singular
locus of the curve $F$ is equal to the set of points $\{ (\alpha_i, 0)
\mid  e_i \geq 2\}$.  Let $P_i$ be the point on $\Spec {T}$ where
$x = \alpha_i$, $y = z = 0$.
  
\begin{proposition}
  \label{prop:4.1}
  In the above context, the following are true.
\begin{enumerate}[(a)]
\item The singular locus of $T$ is equal to the set $\{ P_i \mid 
  e_i \geq 2\}$.
\item $T$ is rational.
\item $\Cl(T) \cong \mathbb Z/D \oplus \mathbb Z^{(v-1)}$.
\item $k^\ast = T^\ast$.
\item If $e_i \geq 2$, the singular point $P_i$ of $T$ is a rational
  double point of type $A_n$, where $n = e_i -1$.
\item The natural homomorphism $\Cl(T_{P_i}) \rightarrow \Cl(\hat
  T_{P_i})$ is an isomorphism. Both groups are cyclic of order $e_i$.
\end{enumerate}
\end{proposition}
\begin{proof}
  The singular locus of $T$ consists of the set of points lying over
  the singular locus of $F$, from which (a) follows.  We have $z^2 - f
  = z^2 - y^2 + p(x) = (z-y)(z+y) + p(x)$.  Setting $u = z-y$ and $t =
  z+y$ defines an isomorphism of $k$-algebras
  \begin{equation}
    \label{eq:400}
    T= \frac{k[x,y,z]}{(z^2-y^2+p(x))} \xrightarrow{\alpha} 
    \frac{k[x,u,t]}{(ut + p(x))}
  \end{equation}
  where $\alpha$ sends $z \mapsto (u+t)/2$, $y \mapsto (t-u)/2$,
  $x\mapsto x$.  Upon inverting $u$ and eliminating $t$, \eqref{eq:400}
  gives rise to the isomorphism
  \begin{equation}
    \label{eq:401}
    \frac{k[x,u,t][u^{-1}]}{(ut + p(x))}
    \xrightarrow{\beta}
    k[x,u,u^{-1}]
  \end{equation}
  where $\beta$ sends $t \mapsto -p(x) u^{-1}$, $u \mapsto u$,
  $x\mapsto x$.  This shows $T$ is rational, which is (b).

  The ring on the right hand side of \eqref{eq:401} is factorial, hence
  so is $T[(z-y)^{-1}]$.  Nagata's Theorem \cite[Theorem~7.1]{F:DCG}
  says the class group of $T$ is generated by the prime divisors that
  contain $z-y$.  There are $v$ minimal primes of $z-y$ in $T$, namely
  $(z-y, \ell_1)$, \dots, $(z-y, \ell_v)$. 
  Let $I_i = (z-y,
  \ell_i)$. In the discrete valuation ring $T_{I_i}$, $z-y =
  -(z+y)^{-1}p(x)$.  Therefore $\ell_i$ generates the maximal ideal
  and the valuation of $z-y$ is $e_i$.  The divisor of $z-y$ on
  $\Spec{T}$ is $\Div(z-y) = e_1I_1 +\dots+ e_vI_v$.
 Using the isomorphism
  \eqref{eq:401}, we know that the group of units of $T[(z-y)^{-1}]$ is
  equal to $k^\ast \times \langle z-y \rangle$.
  Nagata's Theorem
  yields the exact sequence
  \begin{equation}
    \label{eq:402}
    1 \rightarrow T^\ast \rightarrow
    T[(z-y)^{-1}]^\ast \xrightarrow{\Div}
    \bigoplus_{i=1}^v \mathbb Z I_i
    \rightarrow \Cl(T) \rightarrow 0
  \end{equation}
  If $D = \gcd(e_1, \dots, e_v)$, then $\Cl(T) \cong \mathbb Z/D
  \oplus \mathbb Z^{(v-1)}$, which is (c).  Part (d) follows from
  \eqref{eq:402}.

  The natural map $\Cl(T) \rightarrow \Cl(T_{P_i})$ is onto, by
  Nagata's Theorem. For each $j \neq i$, $\ell_j$ is invertible in
  $T_{P_i}$. The group $\Cl(T_{P_i})$ is cyclic and is generated by
  the divisor $I_i$.  The isomorphism $\alpha$ in \eqref{eq:400} shows
  that the complete local ring $\hat T_{P_i}$ is isomorphic to the
  completion of $k[a,b,c]/(c^{n+1} + ab)$ at the maximal ideal $(a, b,
  c)$, where $n = e_i - 1$. The map sends $z-y \mapsto a$, $z+y
  \mapsto b$ and $\ell_i \mapsto c \gamma$, where $\gamma$ is
  invertible. Part (e) follows from \cite[A5]{MR543555}.  The class
  group of $k[a,b,c]/(c^{n+1} + ab)$ is cyclic of order $e_i = n+1$
  and is generated by the ideal $(c,a)$. Therefore, the class group of
  $\hat T_{P_i}$ is cyclic of order $e_i = n+1$ and generated by the
  ideal $I_i$. This implies the composite homomorphism
  \begin{equation}
    \label{eq:403}
    \Cl(T) \rightarrow \Cl(T_{P_i}) \rightarrow \Cl(\hat T_{P_i})
  \end{equation}
  is onto. By Mori's Theorem \cite[Corollary~6.12]{F:DCG}, the second
  map in \eqref{eq:403} is one-to-one, so we get (f).
\end{proof}

In Proposition~\ref{prop:4.2}, let $X = \Spec{T}$ and $P = \{P_1,
\dots, P_v\}$.  By Proposition~\ref{prop:4.1}(a), the singular locus
of $X$ is a (possibly empty) subset of $P$.  By a desingularization of
$P$ we mean a proper, surjective, birational morphism $\tau: \tilde X
\rightarrow X$ such that $\tilde X$ is nonsingular and $\tau$ is an
isomorphism on $\tilde X -E$, where $E = \tau^{-1}(P)$.

\begin{proposition}
  \label{prop:4.2}
  In the above notation, with $X = \Spec{T}$, the following are true
  for any desingularization $\tau : \tilde X \rightarrow X$.
  \begin{enumerate}[(a)]
  \item $0=\B(X) = \HH^2(X,\mathbb G_m)$.
  \item $0= \B(\tilde X) = \HH^2(\tilde X,\mathbb G_m)$.
  \item $0 = \B(X-P)$.
  \end{enumerate}
\end{proposition}
\begin{proof}
  All of the claims follow from results in \cite{F:Bgd}.  Using
  \eqref{eq:401} we see that the Brauer group of the ring
  $T[(z-y)^{-1}]$ is trivial.  If $L$ denotes the quotient field of
  $T$, the natural map $B(T) \rightarrow \B(L)$ is one-to-one and
  factors through $B(T) \rightarrow \B(T[(z-y)^{-1}])$.
\end{proof}
\begin{proposition}
  \label{prop:4.3}
  In the above context, $\Pic{T}$ is isomorphic to the subgroup of
  $\Cl(T)$ generated by the ideal classes $e_1I_1$, \dots, $e_vI_v$.
  It is a free $\mathbb Z$ module of rank $v-1$.
\end{proposition}
\begin{proof}
  By \eqref{eq:403} we know that $\Cl(T) \rightarrow \Cl(\hat T_{P_i})$
  is onto.  Also, the image of $I_i$ in $\Cl(\hat T_{P_j})$ has order
  $e_i$ if $i = j$, and order $1$ otherwise.  Therefore the sequence
  \begin{equation}
    \label{eq:404}
    0 \rightarrow \Pic(T) \rightarrow \Cl(T)
    \rightarrow
    \bigoplus_{i=1}^v
    \Cl(\hat T_{P_i}) \rightarrow 0
  \end{equation}
  of \cite{F:Bgd} is exact.
\end{proof}

\begin{proposition}
  \label{prop:4.4}
  In the above context, $\HH^{1}(G,\Cl(T)) \cong\B(S/R)$.  For all $j
  \geq 0$,
  \[
  \HH^{2j + 1}(G,\Cl(T)) \cong
  \begin{cases}
    (\mathbb Z/2)^{(v)} & \text{if $D\equiv 0\pmod{2}$,}\\
    (\mathbb Z/2)^{(v-1)} & \text{else.}
  \end{cases}
  \]
  and
  \[
  \HH^{2j}(G,\Cl(T)) \cong \Cl(T)^G \cong
  \begin{cases}
    \mathbb Z/2 & \text{if $D\equiv 0\pmod{2}$,}\\
    0 & \text{else.}
  \end{cases}
  \]
\end{proposition}
\begin{proof}
  By Proposition~\ref{prop:4.1}, we are in the context of
  \cite[Theorem~2.7]{F:rBgadp}. Then $\HH^1(T,\Cl(T))$ is isomorphic
  to
  $\B(S/R)$.
  If $\sigma$ is the generator for $G$, then $\sigma$ sends every
  non-identity element of $\Cl(T)$ to its inverse.  This and
  \cite[Theorem~2.4]{F:rBgadp} give the rest.
\end{proof}
\begin{remark}
  \label{rem:4.5}
In Proposition~\ref{prop:4.4}, the isomorphism $\Delta :
\HH^1(T,\Cl(T)) \cong \B(S/R)$ is defined in \cite[\S{2.2}]{F:rBgadp}.
We sketch the construction in our present context.  We restrict our
description to one generator, the class of $I_i$ in $\Cl(T)$.  Notice
that $I_i \sigma(I_i) = (z^2 - y^2, (z-y) \ell_i, (z+y) \ell_i,
\ell_i^2) \subseteq (\ell_i)$. As an $A$-module, $\Delta(I_i)$ is the
direct sum $T \oplus I_i$. The multiplication rule on $\Delta(I_i)$ is
defined on two arbitrary ordered pairs $(a,b)$ and $(c,d)$ by the
formula
\begin{equation}
  \label{eq:405}
  (a,b) (c,d) = \left(
  a c + b \sigma(d) \ell_i^{-1}, 
  b \sigma(c) + ad \right)\ldotp
\end{equation}
This makes $\Delta(I_i)$ into an $A$-algebra. Upon localizing to $S$,
the ideal $I_i$ is projective.  The $R$-algebra $\Delta(I_i)
\otimes_AR$ is a generalized crossed product which is an Azumaya
$R$-algebra \cite{MR0241469}.  It is not too hard to show that the
ideal $I_i = (z-y, \ell_i)$ is generated as an $A$-module by $z-y$ and
$\ell_i$. From this it follows that $I_i$ is a free $A$-module.
Therefore, $\Delta(I_i)$ is a free $A$-module of rank $4$.  Is it true
that every height one prime of $T$ is a free $A$-module?  Using
equation \eqref{eq:405}, the multiplication table for $\Delta(I_i)$ is
constructed in Table~\ref{table:4.1}. Upon restricting to the quotient
field $K$ of $A$, it is clear that
 $\Delta(I_i)
\otimes_AK$ is the symbol algebra $(y^2 - p(x), \ell_i)_2$.
\end{remark}
\begin{table}
    \centering
    \begin{tabular}{|c|c|c|c|c|} \hline
      & $(1,0)$ & $(z,0)$ & $(0,\ell_i)$ & $(0,z-y)$ \\ \hline
      $(1,0)$  & $(1,0)$ & $(z,0)$ & $(0,f_i)$ & $(0,z-y)$ \\\hline
      $(z,0)$  & $(z,0)$ & $(y^2-p(x),0)$ & $(0,z \ell_i)$ & $(0,z(z-y))$ \\\hline
      $(0,\ell_i)$  & $(0,\ell_i)$ & $(0, - z \ell_i)$ & $(\ell_i,0)$
      & $(-(z+y),0)$ \\\hline 
      $(0,z-y)$  & $(0,z-y)$ & $(0,-z(z-y))$ & $(z-y,0)$ & $(-p(x)
      \ell_i^{-1},0)$\\\hline 
    \end{tabular}
    \caption{Multiplication table for $\Delta(I_i)$ in Remark~\ref{rem:4.5}.}
    \label{table:4.1}
  \end{table}

\begin{proposition}
  \label{prop:4.6}
  In the above context, the following are true.
  \begin{enumerate}[(a)]
  \item $\Pic(S) \cong
    \begin{cases}
      \mathbb Z/(D/2) \oplus \mathbb Z^{(v-1)} &\text{if $D\equiv 0 \pmod{2}$,}\\
      \mathbb Z/D \oplus \mathbb Z^{(v-1)} & \text{else.}\\
    \end{cases}$
  \item For all $j \geq 0$,
    \[
    \HH^{2j + 1}(G,\Pic(S)) \cong
    \begin{cases}
      (\mathbb Z/2)^{(v)} & \text{if $D \equiv 0 \pmod{4}$,}\\
      (\mathbb Z/2)^{(v-1)} & \text{else.}
    \end{cases}
    \]
    and
    \[
    \HH^{2j}(G,\Pic(S)) \cong \Pic(S)^G = {_2\Pic(S)} \cong
    \begin{cases}
      \mathbb Z/2 & \text{if $D \equiv 0 \pmod{4}$,}\\
      0 & \text{else.}
    \end{cases}
    \]
  \item $S^\ast =
    \begin{cases}
      k^\ast \times \langle z \rangle\times \langle y-\sqrt{p(x)} \rangle
      & \text{if $D \equiv 0 \pmod{2}$,}\\
      k^\ast \times \langle z \rangle & \text{else.}
    \end{cases}$
  \item $\HH^0(G,S^\ast) = R^\ast = k^\ast \times \langle f \rangle$.
    For all $j > 0$,
    \[
    \HH^{2j-1}(G,S^\ast) = \langle 1\rangle
    \]
    and
    \[
    \HH^{2j}(G,S^\ast) \cong
    \begin{cases}
      \mathbb Z/2 & \text{if $D \equiv 0 \pmod{2}$,}\\
      \langle 1 \rangle & \text{else.}
    \end{cases}
    \]
  \end{enumerate}
\end{proposition}
\begin{proof}
Using \cite[Theorem~2.4]{F:rBgadp}, (b) follows
from (a). Using \cite[Theorem~10.35]{R:IHA}, (d) follows from (c).
By the isomorphism \eqref{eq:401}, the image of $z$ in the ring
$k[x,u,u^{-1}]$ is
$(u^2-p(x))(2u)^{-1}$. Therefore,
\begin{equation}
  \label{eq:406}
  S[(z-y)^{-1}] =
  T[z^{-1},(z-y)^{-1}] \xrightarrow{\beta \alpha}
  k[x,u][u^{-1}, (u^2-p(x))^{-1}]
\end{equation}
is an isomorphism.
There are two cases, which we treat separately.

{\bf Case 1:} $D$ is odd.
Then the affine curve $y^2 - p(x)$ is irreducible
and we are in the context of \cite[Theorem~2.8]{F:rBgadp}. In
particular, $\Cl(T) = \Pic(S)$,
and
$S^\ast = k^\ast \times \langle z \rangle$, which gives the second
halves of (a)
of (c).

{\bf Case 2:} $D$ is even.
Then we can factor $p(x) = q(x)^2$ and we get
\begin{equation}
  \label{eq:407}
  R^\ast = k^\ast 
  \times \langle y-q \rangle
\times \langle y+q \rangle
\end{equation}
The group of units in the ring on the right hand side of
\eqref{eq:406} is equal to
$k^\ast
\times \langle u \rangle
\times \langle u+q \rangle
\times \langle u-q \rangle$.
Since $z-y \mapsto
u$, $z-y+q \mapsto u+q$, and
$z-y-q \mapsto u-q$,
we have
\begin{equation}
  \label{eq:408}
  S[(z-y)^{-1}]^\ast =
k^\ast
\times \langle z-y \rangle
\times \langle z-y+q \rangle
\times \langle z-y-q \rangle
\end{equation}
The rings in \eqref{eq:406} are factorial. We repeat the computation in
\eqref{eq:402}, but with the rings $S$ and $S[(z-y)^{-1}]$. First we
determine the minimal primes of $T$ containing  $z-y$, $z-y+q$, and $z-y-q$.
Notice that 
\begin{gather}
  \label{eq:409}
  (z-y+q) (z-y-q) = 2 z (z-y)\\
  \label{eq:410}
  (z-y)(z+y) = p(x) = \ell_1^{e_1}\dotsm \ell_v^{e_v}
\end{gather}
By \eqref{eq:409}, any prime that contains $z-y+q$ also contains $z$,
or $z-y$.  Let $P_1 = (z-y+q, z) = (z, y-q)$. One checks that $P_1$ is
a height one prime of $T$ and that in the local ring $T_{P_1}$, we
have $\nu_{P_1}(z) = 1$, $\nu_{P_1}(y-q) = 2$, $\nu_{P_1}(z-y+q) = 1$.
Any ideal that contains $z-y+q$ and $z-y$ also contains $q$. There are
$v$ minimal primes of $(z-y,q)$. They are $I_i = (z-y, \ell_i)$ for $i
= 1, \dots, v$.  By \eqref{eq:410}, $\ell_i$ generates the maximal
ideal in the local ring $T_{I_i}$. Thus $\nu_{I_i}(\ell_i) = 1$,
$\nu_{I_i}(z-y) = e_i$, $\nu_{I_i}(q) = e_i/2$.  By \eqref{eq:409},
$e_i = \nu_{I_i}(z-y+q) + \nu_{I_i}(z-y-q)$.  Use the fact that
$\nu_{I_i}(z-y+q) \geq e_i/2$ and $\nu_{I_i}(z-y+q) \geq e_i/2$, to
conclude that $\nu_{I_i}(z-y+q) = \nu_{I_i}(z-y-q) = e_i/2$. The
divisor of $z-y+q$ is $\Div(z-y+q) = P_1 + (e_1/2) I_1 + \dots +
(e_v/2 )I_v$.  Let $P_2 = (z, y+q)$. In the same way we get
$\Div(z-y-q) = P_2 + (e_1/2) I_1 + \dots + (e_v/2) I_v$ and $\Div(z-y)
= P_1 + P_2$.  Since $P_1$ and $P_2$ are not prime ideals in $S$,
Nagata's sequence becomes
\begin{equation}
  \label{eq:411}
  1 \rightarrow S^\ast \rightarrow
  S[(z-y)^{-1}]^\ast \xrightarrow{\Div{}}
  \bigoplus_{i=1}^v \mathbb Z I_i
  \rightarrow \Pic(S) \rightarrow 0
\end{equation}
Evaluate $\Div{}$ on the basis \eqref{eq:408}.  The image of $\Div{}$
is generated by the divisor $(e_1/2) I_1+ \dots + (e_v/2)I_v$.  This
completes part (a).  Using \eqref{eq:411} we also find that $S^\ast$ is
generated over $k^\ast$ by $(z-y+q)(z-y-q)(z-y)^{-1}$ and $(z-y+q)^2
(z-y)^{-1}$.  This and the identities \eqref{eq:409} and \eqref{eq:410}
can be used to show that
\begin{equation}
  \label{eq:412}
  S^\ast =
  k^\ast
  \times \langle z \rangle
  \times \langle y-q \rangle
\end{equation}
which completes (c).
\end{proof}
\begin{proposition}
  \label{prop:4.7}
  In the above context, the following are true.
  \begin{enumerate}[(a)]
  \item $\B(R) \cong
    \begin{cases}
      (\mathbb Q/\mathbb Z)^{(v)}& \text{if $D  \equiv 0 \pmod{2}$,}\\
      (\mathbb Q/\mathbb Z)^{(v-1)}& \text{else.}\\
    \end{cases}$
  \item In the exact sequence \eqref{eq:104}, the image of $\alpha_4$ is
    \[
    \B^\smallsmile(S/R) \cong
    \begin{cases}
      \mathbb Z/2 & \text{if $D \equiv 2 \pmod{4}$,}\\
      0 & \text{else.}
    \end{cases}
    \]
  \item $\B(S/R) ={_2\B(R)}$.
  \item The sequence
    \[
    0 \rightarrow \B(S/R)\rightarrow \B(R)
    \rightarrow \B(S) \rightarrow 0
    \]
    is exact. As a homomorphism of
    abstract groups, the map $\B(R) \rightarrow \B(S)$ is
    ``multiplication by $2$''.
  \end{enumerate}
\end{proposition}
\begin{proof}
  In light of Proposition~\ref{prop:4.4}, to prove (c) it suffices to
  prove (a).  We continue to use the notation established in the proof
  of Proposition~\ref{prop:4.6}.  There are two cases, which we treat
  separately.

  {\bf Case 1:} $D$ is odd.  The map $\alpha_4$ in \eqref{eq:104} is the
  crossed product map.  Since $G$ is cyclic, crossed products are
  symbols of the form $(f,g)_2$, where $g \in R^\ast$. The image of
  $\alpha_4$ is therefore generated by $(f,f)_2$, which is split. This
  is the second half of (b).

  Rearrange the factors of $p(x) =
  \ell_1^{e_1}\dotsm \ell_v^{e_v}$ so that $e_i$ is odd for $1 \leq i
  \leq v_0$, and if $i > v_0$, then $e_i$ is even.  For each $i$,
  write $e_i = 2 q_i + r_i$, where $0 \leq r_i < 2$.  Let $r(x) =
  \ell_1^{r_1}\dotsm \ell_{v_0}^{r_{v_0}}$.  The normalization of $F =
  \Spec{k[x,y]/(y^2 - p(x))}$ is the curve $\tilde F =
  \Spec{k[x,w]/(w^2 - r(x))}$, where the map $\tilde F \rightarrow F$
  is defined by $y \mapsto w \ell_1^{q_1}\dotsm \ell_v^{q_v}$.  Let
  $P_i$ denote the closed point on $F$ where $\ell_i= x-\alpha_i = 0$
  and $y = 0$.  Lying above $P_i$ on $\tilde F$ is the closed set
  defined by $x = \alpha_i$ and $w^2 = r(\alpha_i)$. For $1 \leq i
  \leq v_0$, there is only one point on $\tilde F$ lying over $P_i$.
  For $i > v_0$, there are two points on $\tilde F$ lying over $P_i$.
  By Section~\ref{sec:3.2}, $\HH^1(\tilde F,\mathbb Q/\mathbb
  Z) \cong (\mathbb Q/\mathbb Z)^{(v_0-1)}$. By
  \cite[Corollary~1.6]{F:Bgl},
  \begin{equation}
    \label{eq:413}
    0 \rightarrow  \HH^1(\tilde F,\mathbb Q/\mathbb Z) \rightarrow
    \HH^3_F(\mathbb A^2, \mu) \rightarrow  (\mathbb Q/\mathbb Z)^{(v-v_0)}
    \rightarrow 0
  \end{equation}
  is a split exact sequence. Therefore, $\HH^3_F(\mathbb A^2, \mu)$ is
  a free $\mathbb Q/\mathbb Z$-module of rank $v-1$.  By
  \cite[Lemma~0.1]{F:Bgl}, $\B(R) \cong \HH^3_F(\mathbb A^2, \mu)$,
  which is the second half of (a).

  The Brauer group of the ring on the right hand side of \eqref{eq:406}
  is computed using \cite[Corollary~3.2]{F:Bgl}. It is a free $\mathbb
  Q/\mathbb Z$-module of rank $2v-1$.  Since $S$ is an affine surface,
  it follows from \cite[Lemma~0.1]{F:Bgl} and
  \cite[Corollary~1.6]{F:Bgl} that $\B(S)$ is a free $\mathbb
  Q/\mathbb Z$-module of finite rank.  The reduced closed subset of
  $\Spec{S}$ where $z-y = 0$ is the union of $v$ copies of the
  one-dimensional torus. Therefore, $\HH^1(Z(z-y), \mathbb Q/\mathbb
  Z)$ is a free $\mathbb Q/\mathbb Z$-module of rank $v$.  By
  \cite[Corollary~1.4]{F:Bgl}, the sequence
  \begin{equation}
    \label{eq:414}
    0 \rightarrow \B(S)
    \rightarrow \B( S[(z-y)^{-1}])
    \rightarrow \HH^1(Z(z-y), \mathbb Q/\mathbb Z)
    \rightarrow 0
  \end{equation}
  is exact. All of the groups in \eqref{eq:414} are free $\mathbb
  Q/\mathbb Z$-modules of finite rank. Therefore, \eqref{eq:414} splits
  and $\B(S)$ is a free $\mathbb Q/\mathbb Z$-module of rank $v-1$.
  This proves (d) in Case~1.

  {\bf Case 2:} $D$ is even. Then $p = q^2$.  Write $f_1 = y -q$ and
  $f_2 = y + q$.  Let $F_1 = Z(f_1)$ and $F_2 = Z(f_2)$, both
  nonsingular rational curves in $\mathbb A^2$.  Let $P_i$ be the
  point where $y = 0$ and $\ell_i = x -\alpha_i = 0$.  The
  intersection $F_1 \cap F_2$ is equal to the set of $v$ points $P_1,
  \dots, P_v$.  Since $F_i \cong \mathbb A^1$, $\HH^1(F_i,\mathbb
  Q/\mathbb Z) = 0$.  By \cite[Corollary~1.4]{F:Bgl}, $\B(A[f_1^{-1}])
  = 0$ and $\B(A[f_2^{-1}]) = 0$.  By \cite[Corollary~3.2]{F:Bgl},
  $\B(R) = \B(A[f_1^{-1}, f_2^{-1}])$ is isomorphic to the direct sum
  of $v$ copies of $\mathbb Q/\mathbb Z$, which is (a).

  By \eqref{eq:407}, the image of $\alpha_4$ in \eqref{eq:104} is
  generated by the symbols $(f,f_1)_2 \sim (f_2,f_1)_2$ and $(f,f_2)_2
  \sim (f_2,f_1)_2$, hence is cyclic.  We use
  Theorem~\ref{theorem:1.2} to determine the exponent of $(f_1,f_2)_2$
  in $\B(R)$.  Embed the curves in $\mathbb P^2$, and let $F_0$ be the
  line at infinity.  Let $P_0$ be the point at infinity where $y\neq
  0$ and $x = 0$.  Consider $P_i$, where $1 \leq i \leq v$.  Because
  $F_1$ is nonsingular at $P_i$, we see that the local intersection
  multiplicity at $P_i$ is $(F_1 \cdot F_2)_{P_i} = e_i/2$.  Using
  this, we compute the weighted path in the graph
  $\Gamma(F_1+F_2+F_0)$ associated to the symbol $(f_1,f_2)_2$. Near
  the vertex $P_i$ it looks like
  \[
  F_1 \xrightarrow{e_i/2} P_i \xrightarrow{-e_i/2} F_2
  \]
  If $4 \mid e_i$ for all $i$, then the symbol algebra $(f_1,f_2)_2$
  is split. Otherwise,
$(f_1,f_2)_2$ is an element of order two in 
  the image of $\alpha_4$.
This proves
 (b).  Let $L_i = Z(\ell_i)$. Since $F_1$ is nonsingular at
  $P_i$, the intersection cycle on $\mathbb P^2$ is $F_1\cdot L_i =
  P_i + (d/2 - 1)P_0$. Using this, the weighted path in the graph
  $\Gamma =\Gamma(F_1+F_2+F_0)$ associated to the symbol algebra
  $(f_1 f_2^{-1}, \ell_i)_m$ is
  \begin{equation}
    \label{eq:415}
    F_1 \rightarrow P_i \rightarrow F_2 \rightarrow P_0 \rightarrow F_1
  \end{equation}
  For $i = 1, \dots, v$, the cycles \eqref{eq:415} make up a basis for
  $\HH_1(\Gamma, \mathbb Z/m)$. Using Theorem~\ref{theorem:1.2}, this
  proves that the algebras $(f_1 f_2^{-1}, \ell_i)_m$ form a basis for
  ${_m\B(R)}$.  Upon extending scalars to $S$, we have
  \begin{equation}
    \label{eq:416}
    (f_1 f_2^{-1}, \ell_i)_m \sim
    ((y^2-p)/(y+q)^2, \ell_i)_m \sim
    (z/(y+q), \ell_i)_m^2
  \end{equation}
  The image of $\B(R) \rightarrow \B(S)$ is divisible by two.  To
  complete part (d), as in the previous case it is enough to show that
  $\B(S)$ is a direct sum of $v$ copies of $\mathbb Q/\mathbb Z$.  We
  use sequence \eqref{eq:414}.  As in the previous case, the group
  $\HH^1(Z(z-y), \mathbb Q/\mathbb Z)$ is a free $\mathbb Q/\mathbb
  Z$-module of rank $v$.  The ring in the right hand side of
  \eqref{eq:406} is isomorphic to $R[y^{-1}]$. Use
  \cite[Corollary~3.2]{F:Bgl} to see that $\B(R[y^{-1}])$ is
  isomorphic to $\B(R) \oplus (\mathbb Q/\mathbb Z)^{(v)}$.
  Therefore, the rank of $\B(S[(z-y)^{-1}])$ is equal to $2v$. By
  \eqref{eq:414}, $\B(S)$ has rank $v$, completing the proof.
\end{proof}
\begin{remark}
  Consider the affine hyperelliptic curve $F = Z\left(y^2 -
    (x-\alpha_1)^{e_1} \dotsm (x-\alpha_v)^{e_v}\right)$.  For the
  affine double plane $X \rightarrow \mathbb A^2$ that ramifies along
  $F$, and the invariants of $X$ computed in Section~\ref{sec:4}, the
  roots $\alpha_1, \dots, \alpha_v$ do not seem to matter, whereas the
  number of roots $v$ and multiplicities $e_1, \dots, e_v$ of the
  roots play a role.  What invariants of $X$ depend on the actual
  roots $\alpha_1, \dots, \alpha_v$?
\end{remark}

\section{Divisors on hyperelliptic curves}
\label{sec:3}
\subsection{Divisors on a projective hyperelliptic curve}
\label{sec:3.1}
In this section $Y$ denotes a nonsingular integral projective
hyperelliptic curve and $\pi : Y \rightarrow \mathbb P^1$ is a double
cover with ramification locus $Q = \{Q_1, \dots, Q_n\}$ on $Y$.  Using
the Riemann-Hurwitz Formula \cite[Corollary~IV.2.4]{H:AG}, it follows
that $n$ is an even integer greater than or equal to $4$ and the genus
of $Y$ is $g = (n-2)/2$.  By Kummer Theory, the group $\HH^1(Y,
\mathbb Q/\mathbb Z) \cong (\mathbb Q/\mathbb Z)^{(n-2)}$ classifies
the cyclic Galois covers of $Y$ and can be identified with the torsion
subgroup of $\Pic(Y)$ \cite[Proposition~III.4.11]{M:EC}.  Let $P_i$
denote the image $\pi(P_i)$ on $\mathbb P^1$ and $P = \{P_1, \dots,
P_n\}$.  Consider $U = \mathbb P^1 - P$ and $X = Y - Q$.  Then $\pi :
X \rightarrow U$ is a quadratic Galois cover.  The affine coordinate
ring $\mathcal O(U)$ is isomorphic to $k[x][f^{-1}]$, where $f$
factors into $n-1$ distinct linear polynomials. So $\Pic(U) = 0$.
There is a basis for $U^\ast/k^\ast = \HH^0(U,\mathbb G_m)/k^\ast$
corresponding to the principal divisors $P_1 - P_n$, \dots, $P_{n-1} -
P_n$ on $\mathbb P^1$.  Denote by $f_i$ an element of $K(U)$ such that
the divisor of $f_i$ on $\mathbb P^1$ is $\Div(f_i) = P_i- P_n$. Then
$U^\ast/k^\ast = \langle f_1\rangle \times \dotsm \times \langle
f_{n-1}\rangle$.  By Kummer Theory,
\begin{equation}
  \label{eq:139}
  U^\ast / (U^\ast)^2 \cong
  \HH^1(U, \mathbb Z/2) \cong (\mathbb Z/2)^{(n-1)}\ldotp
\end{equation}
Consider the natural maps
\begin{equation}
  \label{eq:140}
  \begin{CD}
    @. \HH^1(U,\mathbb Z/2) \\
    @. @VV{\pi^\ast}V  \\
    \HH^1(Y,\mathbb Z/2) @>{\beta}>> \HH^1(X,\mathbb Z/2) \\
  \end{CD}
\end{equation}
By \cite[Lemma~3.11]{F:gouoav}, the images $\pi^\ast(f_i)$ generate a
$\mathbb Z/2$-submodule of $\HH^1(X,\mathbb Z/2)$ of rank $n-2$ and
each image $\pi^\ast(f_i)$ is in the image of $\beta$.  Therefore the
image of $\beta$ is generated by $\pi^\ast(f_1), \dots,
\pi^\ast(f_{n-1})$.  It follows that ${_2\Pic(Y)}$ is generated by the
divisors $Q_1 - Q_n$, \dots, $Q_{n-1} - Q_n$.  The kernel of
$\pi^\ast$ is cyclic and corresponds to the quadratic Galois cover
$\pi : X \rightarrow U$.  On the function fields, this corresponds to
adjoining the square root of $f_1 \dotsm f_{n-1}$ to $K(\mathbb P^1)$.
If $z^2 = f_1 \dotsm f_{n-1}$, then $\Div(z) = Q_1 + Q_2 + \dots
Q_{n-1} - (n-1) Q_n$ is a principal divisor on $Y$.  This shows that
\begin{equation}
  \label{eq:145}
  \{Q_1 - Q_n, \dots, Q_{n-2} -
  Q_n\}
\end{equation}
is a $\mathbb Z/2$-basis for ${_2\Pic(Y)}$, and the group of units on
$X$ is
\begin{equation}
  \label{eq:144}
  \mathcal O^\ast(X) = k^\ast \times \langle z \rangle
  \times \langle f_{2} \rangle \dotsm
  \times \langle f_{n-1} \rangle\ldotp
\end{equation}

\subsection{Divisors on an affine hyperelliptic curve}
\label{sec:3.2}
In this section we consider an affine hyperelliptic curve. Let $n \geq
4$ be an integer.  Let $\lambda_1, \dots, \lambda_n$ be distinct
elements of $k$ and set $f(x) = (x-\lambda_1)\dotsm (x-\lambda_n)$.
Let $X = Z(y^2 - f(x))$, an affine hyperelliptic curve in $\mathbb
A^2$. Let $\pi : X \rightarrow \mathbb A^1$ be the morphism induced by
$k[x] \rightarrow \mathcal O(Y)$. Let $Y$ be the complete nonsingular
model for $X$ and let $\pi : Y \rightarrow \mathbb P^1$ be the
extension of $\pi$.  Let $Q_i$ denote the point on $X$ (and on $Y$)
where $y = x-\lambda_i = 0$.  Let $P_i = \pi(Q_i)$.  Employing the
computations from Section~\ref{sec:3.1}, there are two cases.

Case 1: $n$ is even.  Then $Y - X$ consists of two points, say
$Q_{01}$ and $Q_{02}$, and $\pi : Y \rightarrow \mathbb P^1$ ramifies
only on the set of $n$ points $Q = \{Q_1, \dots, Q_n\}$.  The genus of
$Y$ is $(n-2)/2$.  By \eqref{eq:145}, the divisor classes $Q_1 - Q_n$,
\dots, $Q_{n-2} - Q_n$ form a $\mathbb Z/2$-basis for $\HH^1(Y,\mathbb
Z/2) = {_2\Pic(Y)}$.  Because $X$ is affine, $\HH^2(X,\mathbb Z/2) =
0$.  Because $Y$ is projective, $\HH^2(Y,\mathbb Z/2) = \mathbb Z/2$.
Because $Y$ is nonsingular,
\[
\HH^p_{Q_{01}+Q_{02}}(Y,\mathbb Z/2) =
\begin{cases}
  0 & \text{if $p < 2$}\\
  \mathbb Z/2 \oplus \mathbb Z/2 & \text{if $p = 2$}
\end{cases}
\]
by Cohomological Purity \cite[Theorem~VI.6.1]{M:EC}.  The sequence of
cohomology with supports in $Y - X$ is
\begin{equation}
  \label{eq:141}
  0 \rightarrow 
  \HH^1(Y,\mathbb Z/2)
  \rightarrow
  \HH^1(X,\mathbb Z/2)
  \rightarrow
  \HH^2_{Q_{01}+Q_{02}}(Y,\mathbb Z/2)
  \rightarrow
  \HH^2(Y,\mathbb Z/2) \rightarrow 0
\end{equation}
which implies
\begin{equation}
  \label{eq:146}
  \HH^1(X,\mathbb Z/2) \cong (\mathbb Z/2)^{(n-1)}\ldotp    
\end{equation}
Let $P = \{P_1, \dots, P_n\}$, $V = Y - Q$, and $U = \mathbb A^1- P$.
The group of units on $U$ is $U^\ast = k^\ast \times \langle
x-\lambda_1\rangle \times \dotsm \times \langle x-\lambda_n\rangle$.
By Kummer Theory, $\HH^1(U,\mathbb Z/2)$ is isomorphic to
$U^\ast/U^{\ast 2}$.  The counterpart of diagram \eqref{eq:140} for
the Galois cover of affine curves $\pi : V \rightarrow U$ is
\begin{equation}
  \label{eq:142}
  \begin{CD}
    @.   \HH^1(U,\mathbb Z/2) \\
    @.   @VV{\pi^\ast}V  \\
    \HH^1(X,\mathbb Z/2) @>{\beta}>> \HH^1(V,\mathbb Z/2)
  \end{CD}
\end{equation}
The image of $\pi^\ast$ has $\mathbb Z/2$-rank $n-1$ and is contained
in the image of $\beta$, which also has $\mathbb Z/2$-rank $n-1$. This
proves that the image of $\pi^\ast$ is equal to the image of $\beta$.
This proves that the quadratic Galois extensions of $X$ are all of the
form
\begin{equation}
  \label{eq:143}
  \mathcal O(X) \otimes_{k[x]} k[x]/(y^2 = (x-\lambda_1)^{e_1}
  \dotsm (x-\lambda_n)^{e_n})
\end{equation}
where each $e_i$ is $0$ or $1$. The quadratic extension in
\eqref{eq:143} is split if and only if $e_1 = e_2 = \dots = e_n$.
 
Case 2: $n$ is odd.  Then $Y - X = Q_0$ is a single point, and $\pi$
ramifies at the points $Q_0, Q_1, \dots, Q_n$.  The genus of $Y$ is
$(n-1)/2$. The divisor classes $Q_1 - Q_0$, \dots, $Q_{n-1} - Q_0$
form a $\mathbb Z/2$-basis for ${_2\Pic(Y)}$.  The Gysin sequence
\cite[Remark~VI.5.4]{M:EC}
\begin{equation}
  \label{eq:148}
  0 \rightarrow \HH^1(Y,  \mathbb Z/2)
  \rightarrow \HH^1(X,  \mathbb Z/2)
  \rightarrow
  \HH^2_{Q_0}(Y,  \mathbb Z/2)
  \rightarrow
  \HH^2(Y,  \mathbb Z/2)
  \rightarrow 0
\end{equation}
shows that
\begin{equation}
  \label{eq:147}
  \HH^1(Y, \mathbb Z/2) \cong
  \HH^1(X, \mathbb  Z/2) \cong (\mathbb Z/2)^{(n-1)}\ldotp   
\end{equation}
Therefore, $Q_1$, \dots, $Q_{n-1}$ form a $\mathbb Z/2$-basis for
$\HH^1(X, \mathbb Z/2) = {_2\Pic(X)}$.

\begin{proposition}
  \label{prop:3.1}
  As above, let $n \geq 3$ and $X$ the affine hyperelliptic curve
  defined by $y^2 = \prod_{i=1}^n (x - \lambda_i)$. Let $\sigma : X
  \rightarrow X$ be the automorphism defined by $y \mapsto -y$. Let $G
  = \langle \sigma\rangle$.  Then for all $i \geq 0$,
  \begin{enumerate}[(a)]
  \item $\HH^{2i+1}(G, \Pic{X}) = 0$ and
  \item $\HH^{2i}(G, \Pic{X}) = \left(\Pic{X}\right)^G = {_2\Pic{X}}
    \cong (\mathbb Z/2)^{(r)}$.  If $n$ is odd, then $r$ is equal to
    $n-1$. If $n$ is even, then $r$ is either $n-1$ or $n-2$.
  \end{enumerate}
\end{proposition}
\begin{proof}
  Every prime divisor on $X$ is mapped by $\sigma$ to its inverse.
  Since $X$ is not rational, the only divisors that are fixed by $G$
  are the elements of order two in $\Pic{X}$.  That is,
  $\left(\Pic{X}\right)^G = {_2\Pic{X}}$.  Following the notation of
  \cite[pp.  296--297]{R:IHA}, let $N = \sigma +1$ and $D = \sigma
  -1$.  It follows that $N : \Pic{X} \rightarrow \Pic{X}$ is the zero
  map and $D : \Pic{X} \rightarrow \Pic{X}$ is multiplication by $-2$.
  If $_N\Pic{X}$ denotes the kernel of $N$, then by
  \cite[Theorem~10.35]{R:IHA} we have
  \[
  \HH^{2i}(G,\Pic{X}) = \frac{\left(\Pic{X}\right)^G}{N \Pic{X}} =
  \left(\Pic{X}\right)^G,
  \]
  the first part of (b), and
  \[
  \HH^{2i+1}(G,\Pic{X}) = \frac{_N\Pic{X}}{D \Pic{X}} = \Pic{X}\otimes
  \mathbb Z/2\ldotp
  \]
  By Kummer Theory, $\Pic{X}\otimes \mathbb Z/2 \rightarrow \HH^2(X,
  \mu_2)$ is one-to-one.  Since $X$ is an affine curve, $\HH^2(X,
  \mu_2)=0$, which proves (a).  The rest of (b) follows from
  \cite[Lemmas~3.1, 3.2]{F:gouoav}.
\end{proof}


\def\cfudot#1{\ifmmode\setbox7\hbox{$\accent"5E#1$}\else
  \setbox7\hbox{\accent"5E#1}\penalty 10000\relax\fi\raise 1\ht7
  \hbox{\raise.1ex\hbox to 1\wd7{\hss.\hss}}\penalty 10000 \hskip-1\wd7\penalty
  10000\box7}
\providecommand{\bysame}{\leavevmode\hbox to3em{\hrulefill}\thinspace}
\providecommand{\MR}{\relax\ifhmode\unskip\space\fi MR }
\providecommand{\MRhref}[2]{%
  \href{http://www.ams.org/mathscinet-getitem?mr=#1}{#2}
}
\providecommand{\href}[2]{#2}

\end{document}